\documentclass[leqno]{article}

\pdfoutput=1

\usepackage[english]{babel}

\usepackage[letterpaper,top=2cm,bottom=2cm,left=3.5cm,right=3.5cm,
marginparwidth=1.75cm]{geometry}

\usepackage{amsmath,amsthm,amssymb,tikz,tikz-cd,ytableau,relsize}
\usepackage{stmaryrd,graphicx,xcolor}
\usepackage[colorlinks=true, allcolors=blue]{hyperref}
\usetikzlibrary{calc,tikzmark}
\usepackage{tcolorbox}
\usepackage[framemethod=tikz]{mdframed}
\usepackage{lipsum}
\usepackage{tikz}
\usepackage{tkz-fct} 
\usepackage[toc,page]{appendix}
\usepackage{amsthm}

\usetikzlibrary{positioning}
\usetikzlibrary{patterns}
\usetikzlibrary{arrows}
\usetikzlibrary{cd}
\tikzset{>=stealth}

\newcommand{\vf}{\vfill\end{document}}

\newtheorem{Theorem}[equation]{Theorem}
\newtheorem{Corollary}[equation]{Corollary}
\newtheorem{Property}[equation]{Property}
\newtheorem{Lemma}[equation]{Lemma}
\newtheorem{Fact}[equation]{Fact}

\newtheorem*{Heuristic}{Heuristic}

\theoremstyle{definition}

\newtheorem{Definition}[equation]{Definition}
\newtheorem{Notation}[equation]{Notation}
\newtheorem{Example}[equation]{Example}
\newtheorem{Remark}[equation]{Remark}
\newtheorem*{Question}{Question}

\numberwithin{equation}{section}

\title{Computation of the multiplicities of zigzags}

\date{}

\author{Victor Chen}

\newtcolorbox{Def}[1]{colback=black!5!white,colframe=green!75!black,title=#1}

\newtcolorbox{Prop}[1]{colback=blue!5!white,colframe=red!75!black,title=#1}

\newcommand{\case}[2]{%
 {\begin{tikzpicture}[baseline=1.5ex, inner sep=0pt, outer sep=0pt, remember picture, scale=0.8]
\coordinate (A) at (0, #2);
\draw (0, 0) -- (0,#2) -- (1,#2) -- (1,#2 + 1) -- (0,#2 + 1) -- (0, 0);
\end{tikzpicture}}
\ifstrequal{#1}{0}%
    { {\begin{tikzpicture}[overlay, remember picture, scale=0.8]
\fill [black]
    (A)+(.5,.5) circle [radius=0.1];
\end{tikzpicture}}}% #1 = #2 -> [#1]
    {}% [#1,#2]
  \ifstrequal{#1}{1}%
    { {\begin{tikzpicture}[overlay, remember picture, scale=0.8]
\draw [thick,domain=0:90] plot ({-1 + 1/2*cos(\x)}, {-.3+1/2*sin(\x)});
\end{tikzpicture}}}% #1 = #2 -> [#1]
    {}% [#1,#2]
\ifstrequal{#1}{2}%
    { {\begin{tikzpicture}[overlay, remember picture, scale=0.8]
\draw[thick] (A)+(.5, 0) -- +(.5, .5);
\fill [black]
    (A)+(.5,.5) circle [radius=0.1];
\end{tikzpicture}}}% #1 = #2 -> [#1]
    {}% [#1,#2]
\ifstrequal{#1}{3}%
    { {\begin{tikzpicture}[overlay, remember picture, scale=0.8]
\draw[thick] (A)+(0, .5) -- +(.5, .5);
\fill [black]
    (A)+(.5,.5) circle [radius=0.1];
\end{tikzpicture}}}% #1 = #2 -> [#1]
    {}% [#1,#2]
\ifstrequal{#1}{4}%
    { {\begin{tikzpicture}[overlay, remember picture, scale=0.8]
\draw[thick] (A)+(.5, 1) -- +(.5, .5);
\fill [black]
    (A)+(.5,.5) circle [radius=0.1];
\end{tikzpicture}}}% #1 = #2 -> [#1]
    {}% [#1,#2]
\ifstrequal{#1}{5}%
    { {\begin{tikzpicture}[overlay, remember picture, scale=0.8]
\draw[thick] (A)+(1, .5) -- +(.5, .5);
\fill [black]
    (A)+(.5,.5) circle [radius=0.1];
\end{tikzpicture}}}% #1 = #2 -> [#1]
    {}% [#1,#2]
  \ifstrequal{#1}{6}%
    { {\begin{tikzpicture}[overlay, remember picture, scale=0.8]
\draw [thick,domain=0:90] plot ({ - 1/2*cos(\x)}, {.7-1/2*sin(\x)});
\end{tikzpicture}}}% #1 = #2 -> [#1]
    {}% [#1,#2]
}

\begin{document}

\maketitle

\begin{abstract}
In this note, we explore various cohomological invariants on double complexes with the aim of finding their decomposition into irreducible parts, which are of square and zigzag shape. By studying the growth rate of the number of invariants given by the multiplicities of zigzags in the double complex of an $n$-dimensional complex manifold, we show that the De Rham, Dolbeault, Bott-Chern, Aeppli, and Varouchas cohomologies do not suffice to distinguish non-isomorphic double complexes. We also describe the zigzags counted by the Bigolin cohomology, and show how their dimensions are related to the multiplicities of odd zigzags.

A special class of complex manifolds is given by the nilmanifolds. For a nilmanifold, the double complex of left-invariant forms is quasi-isomorphic to the double complex of differential forms. In dimension $6$, we compute the double complex of forms of two nilmanifolds having the same Betti, Hodge and Bott-Chern numbers, but whose double complexes are non-isomorphic. We also compute the double complexes of a subclass of almost abelian nilmanifolds, which exist in any dimension.
\end{abstract}

\tableofcontents

\section{Introduction}
Let $M$ be a complex manifold. 
We denote by $A^k(M)$ the space of smooth $k$-forms on $M$ and by $A^{p, q}(M)$ the space of smooth $(p, q)$-forms on $M$. 
The exterior differential $d$ can be written as the sum of the Dolbeault operators $\partial$ and $\overline{\partial}$:
\[
d
=
\partial
+
\overline{\partial}.
\]
It is classical that $d^2 = 0$, $\partial^2 = 0$, $\overline{\partial}^2=0$ and $\partial \overline{\partial} = \overline{\partial} \partial$. These conditions define a {\sl double complex} $(A^{\bullet, \bullet}, \partial, \overline{\partial})$:
\[
\begin{tikzcd}[scale=0.3em]
\vdots & \vdots & \vdots & \empty \\
A^{0, 2}(X) \arrow[u] \arrow[r]  & A^{1, 2}(X) \arrow[u] \arrow[r] & A^{2, 2}(X) \arrow[u] \arrow[r] &\cdots \\
A^{0, 1}(X) \arrow[u] \arrow[r]  & A^{1, 1}(X) \arrow[u] \arrow[r] & A^{2, 1}(X) \arrow[u] \arrow[r] &\cdots \\
A^{0, 0}(X) \arrow[u] \arrow[r]  & A^{1, 0}(X) \arrow[u] \arrow[r] & A^{2, 0}(X) \arrow[u] \arrow[r] &\cdots \\
\end{tikzcd}
\]

We will consider two types of numerical invariants that one can attach to a double complex. 
The first type is obtained by decomposing a double complex.
In section~\ref{deux}, we recall that a double complex can be decomposed 
as a direct sum of irreducible double complexes, 
which are of square and zigzag shape~(\cite{Ste}, \cite{KQ}). 
\[
\text{A square :}
\begin{tikzcd}
\mathbb{C} \arrow[r, "1"] 
& \mathbb{C} \\
\mathbb{C} \arrow[r, "1"] \arrow[u, "1"]
& \mathbb{C} \arrow[u, "-1"] 
\end{tikzcd},
\]
\[
\text{Zigzags :}
\begin{tikzcd}[row sep=small, column sep = small]
\mathbb{C} & & \\
\mathbb{C} \arrow[u] \arrow[r] & \mathbb{C} & \\
& \mathbb{C} \arrow[u]
\end{tikzcd},
\quad \quad
\begin{tikzcd}[row sep=small, column sep = small]
\mathbb{C} & & \\
\mathbb{C} \arrow[u] \arrow[r] & \mathbb{C} & \\
& \mathbb{C} \arrow[u] \arrow[r] & \mathbb{C}
\end{tikzcd},
\quad \quad
\begin{tikzcd}[row sep=small, column sep = small]
\mathbb{C}\arrow[r] & \mathbb{C} & \\
& \mathbb{C} \arrow[u] \arrow[r] & \mathbb{C} \\
& & \mathbb{C} \arrow[u]
\end{tikzcd},
\ 
\cdots
\]
We will distinguish zigzags having an even number of components from zigzags having an odd number of components, and call them {\sl even zigzags} and {\sl odd zigzags} respectively.

For a compact complex manifold, the zigzags appearing in the decomposition of the associated double complex have finite multiplicities. Thus, these multiplicities will be the first type of numerical invariants that we shall consider.
%%%%%%%%%%%%%%%%%%%%%%%%%%%%%%%%%%%%%%%%%%

The second type of numerical invariant are given by cohomologies.
The cohomology associated to the total complex is the {\sl De Rham cohomology}. 
For a compact complex manifold, the cohomology spaces are finite dimensional, and their dimensions
$b^k$ are called the {\sl Betti numbers}. These numbers are topological invariants of the manifold.

The cohomology associated to the vertical complex
\[
A^{p, 0}
\xrightarrow{\overline{\partial}}
A^{p, 1}
\xrightarrow{\overline{\partial}}
\cdots ,
\]
is called the {\sl Dolbeault cohomology}. The dimension $h^{p, q} = \dim H_{\overline{\partial}}^{p, q}(M)$ is called the $(p, q)$-{\sl Hodge number}, and the data of all these numbers is referred to as the {\sl Hodge diamond}:
\[
\aligned
&\quad \quad \quad h^{n, n} \\
&\quad h^{n, n-1} \ h^{n-1, n} \\
&h^{n, n-2} \ h^{n-1, n-1} \ h^{n-2, n} \\
&\quad \quad \quad \vdots \\
&\quad \ \ h^{0, 1} \ h^{1, 0} \\
&\quad \quad \quad h^{0, 0}
\endaligned
\]
Hodge numbers are invariants of the complex structure of $M$. The relation between the topological invariants $b^k$ and the Hodge numbers $h^{p, q}$ is given by the {\sl Frölicher spectral sequence} (\cite{Fro}).

Other cohomological invariants can be considered: there is the {\sl Bott-Chern} and the {\sl Aeppli} cohomologies. Other cohomologies are given by the {\sl Varouchas} cohomologies~(\cite{JV}) and the Bigolin cohomology~(\cite{Sch}, \cite{Ste3}, \cite{Piovani}).

All cohomologies commute with direct sums and they vanish on squares. Thus, the multiplicities of the zigzags in a double complex determine all cohomologies. We are interested in the reverse question.

\begin{Question}
Fixing a collection of cohomological invariants of a double complex, to what extent do they determine its decomposition into irreducible summands?
\end{Question}

The values of cohomological invariants give linear equations involving the zigzags.
It is interesing, for heuristical reasons, to know the growth rates of the different numerical invariants.

\begin{Property}
The number of invariants considered above grows as follows in $n = \dim_{\mathbb{C}}X$.
\begin{itemize}
\item The Betti numbers $b^k$, $O(n)$.
\item The Hodge and conjugate Hodge numbers $h_{\overline{\partial}}$ and $h_{\partial}$, $O(n^2)$.
\item The Bott-Chern and Aeppli cohomologies $h_{\text{BC}}$ and $h_{\text{A}}$, $O(n^2)$.
\item The Varouchas cohomologies, $O(n^2)$.
\item The Frölicher spectral sequences and the Bigolin numbers, $O(n^3)$.
\item Multiplicities of zigzags, $O(n^3)$.
\end{itemize} 
\end{Property}

In fact, for a compact complex manifold, the zigzags in the decomposition of the double complex must respect some symmetries and cannot have components in a corner. Under these conditions, we compute the exact number of linearly independent zigzag multiplicities. However, we do not know if all these zigzags can be obtained by formal sums of double complexes coming from complex manifolds.
\begin{Heuristic}
\[
\# \, \text{linearly independent zigzags}
=
\frac{1}{4}
\Big(
\frac{4n^3 - 4n}{3}
+5n^2
- 5
+ \delta_{n \in 2\mathbb{Z}}
\Big).
\]
\end{Heuristic}
 It is noteworthy that the number grows in $n^3$, whereas the number of linear equations given by the cohomological invariants (except for the Bigolin cohomology and the Frölicher spectral sequence which also have $O(n^3)$ equations) is $O(n^2)$.

In section~\ref{loc_sim}, we take a deeper look at double complexes sharing the same Hodge, Bott-Chern, Aeppli and Varouchas cohomologies (\cite{JV}). Such double complexes will be referred to as being {\sl locally similar}, as their local information in each bidegree coincide. From the asymptotics of the invariants, locally similar double complexes having different zigzag multiplicities must exist. 
\begin{Theorem}
Starting from formal dimension 4, there are locally similar double complexes with different zigzag multiplicities.
\end{Theorem}
We also give an abstract construction 
of such double complexes which also have the same Betti numbers and Frölicher spectral sequences.

Moreover, despite the similar growth rates, the Bigolin numbers are not sufficient to determine the multiplicities of all zigzags. 
Computations involving the Bigolin complex and other cohomological invariants in dimension $3$ 
have been done in (\cite{Piovani}). For some fixed bidegree $(p, q)$ and total degree $k$, 
we describe the zigzags which are counted by the Bigolin number $b^k_{p, q}$. 
Ultimately, we prove the following:
\begin{Theorem}
The Bigolin numbers determine the multiplicities of all odd zigzags. On the other hand, starting in formal dimension $n = 4$, there exist double complexes with the same Bigolin numbers but with different multiplicities of even zigzags.
\end{Theorem}
In dimension $4$, we give an example of two double complexes sharing the same Bigolin, Betti, Hodge, Bott-Chern, Aeppli and Varouchas numbers.

If instead of looking at all compact complex manifolds we consider only a certain geometric class, then we can hope for better results. We illustrate this with compact complex nilmanifolds in section~\ref{lie}.
It is conjectured that the double complex of differential forms for nilmanifolds 
is quasi-isomorphic to the double complex of left-invariant differential forms. In many cases this is known to be true (\cite{CF}), and therefore the double complex of differential forms 
is quasi-isomorphic to the {\sl Chevalley-Eilenberg complex} of the Lie algebra $\mathfrak{g}$. 
In real dimension $6$, nilpotent Lie algebras admitting complex structures have been classified in~\cite{Sal}, and their complex structures have been classified up to equivalence in~\cite{COUV}. The Betti and Bott-Chern numbers of these structures are computed in~\cite{Ang2}. Interestingly, the Lie algebra $\mathfrak{h}_5$ admits two complex structures $J_1$ and $J_2$ which have the same Betti, Hodge and Bott-Chern numbers, but whose double complexes are non-isomorphic:

\begin{Theorem}
On the real $6$-dimensional Lie algebra $h_5$, there are two nilpotent complex structures which have the same Betti, Hodge and Bott-Chern numbers but different multiplicities of zigzags.
%%%%%%%%%%%%%%%%%%%%%%%%%%%%%%%%%%%
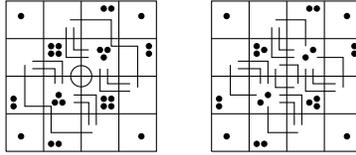
\[
\begin{tikzpicture}[baseline=5.8ex, inner sep=0pt, outer sep=0pt, remember picture]
\draw (0, 0) -- (2, 0) -- (2, 2) -- (0, 2) -- (0, 0);
\draw (0, .5) -- (2, .5);
\draw (0, 1) -- (2, 1);
\draw (0, 1.5) -- (2, 1.5);
\draw (.5, 0) -- (.5, 2);
\draw (1, 0) -- (1, 2);
\draw (1.5, 0) -- (1.5, 2);
\draw (0.25, 1.15) -- (0.25, .6) -- (.6, .6) -- (.6, .25) -- (1.15, .25);
\draw (1.75, .85) -- (1.75, 1.4) -- (1.4, 1.4) -- (1.4, 1.75) -- (.85, 1.75);
\filldraw (.1, .6) circle (1pt);
\filldraw (.1, .7) circle (1pt);
\filldraw (1.9, 1.4) circle (1pt);
\filldraw (1.9, 1.3) circle (1pt);
\filldraw (.6, .1) circle (1pt);
\filldraw (.7, .1) circle (1pt);
\filldraw (1.4, 1.9) circle (1pt);
\filldraw (1.3, 1.9) circle (1pt);
\filldraw (.65, .65) circle (1pt);
\filldraw (.7, .75) circle (1pt);
\filldraw (.75, .65) circle (1pt);
\filldraw (1.35, 1.35) circle (1pt);
\filldraw (1.3, 1.25) circle (1pt);
\filldraw (1.25, 1.35) circle (1pt);
\draw (.35, 1.1) -- (.65, 1.1) -- (.65, .9);
\draw (.35, 1.2) -- (.75, 1.2) -- (.75, .9);
\draw (1.1, .35) -- (1.1, .65) -- (.9, .65);
\draw (1.2, .35) -- (1.2, .75) -- (.9, .75);
\draw (.9, 1.65) -- (.9, 1.35) -- (1.1, 1.35);
\draw (.8, 1.65) -- (.8, 1.25) -- (1.1, 1.25);
\draw (1.65, .9) -- (1.35, .9) -- (1.35, 1.1);
\draw (1.65, .8) -- (1.25, .8) -- (1.25, 1.1);
\filldraw (.6, 1.4) circle (1pt);
\filldraw (.6, 1.3) circle (1pt);
\filldraw (.7, 1.4) circle (1pt);
\filldraw (.7, 1.3) circle (1pt);
\filldraw (1.4, .6) circle (1pt);
\filldraw (1.4, .7) circle (1pt);
\filldraw (1.3, .6) circle (1pt);
\filldraw (1.3, .7) circle (1pt);
\draw (1, 1) circle (4pt);
\filldraw (.2, .2) circle (1pt);
\filldraw (1.8, .2) circle (1pt);
\filldraw (.2, 1.8) circle (1pt);
\filldraw (1.8, 1.8) circle (1pt);
\end{tikzpicture}
\quad \quad
\begin{tikzpicture}[baseline=5.8ex, inner sep=0pt, outer sep=0pt, remember picture]
\draw (0, 0) -- (2, 0) -- (2, 2) -- (0, 2) -- (0, 0);
\draw (0, .5) -- (2, .5);
\draw (0, 1) -- (2, 1);
\draw (0, 1.5) -- (2, 1.5);
\draw (.5, 0) -- (.5, 2);
\draw (1, 0) -- (1, 2);
\draw (1.5, 0) -- (1.5, 2);
\draw (.35, 1.1) -- (.65, 1.1) -- (.65, .9);
\draw (.35, 1.2) -- (.75, 1.2) -- (.75, .9);
\draw (1.1, .35) -- (1.1, .65) -- (.9, .65);
\draw (1.2, .35) -- (1.2, .75) -- (.9, .75);
\draw (.9, 1.65) -- (.9, 1.35) -- (1.1, 1.35);
\draw (.8, 1.65) -- (.8, 1.25) -- (1.1, 1.25);
\draw (1.65, .9) -- (1.35, .9) -- (1.35, 1.1);
\draw (1.65, .8) -- (1.25, .8) -- (1.25, 1.1);
\filldraw (.2, .2) circle (1pt);
\filldraw (1.8, .2) circle (1pt);
\filldraw (.2, 1.8) circle (1pt);
\filldraw (1.8, 1.8) circle (1pt);
\filldraw (.1, .6) circle (1pt);
\filldraw (.1, .7) circle (1pt);
\filldraw (1.9, 1.4) circle (1pt);
\filldraw (1.9, 1.3) circle (1pt);
\filldraw (.6, .1) circle (1pt);
\filldraw (.7, .1) circle (1pt);
\filldraw (1.4, 1.9) circle (1pt);
\filldraw (1.3, 1.9) circle (1pt);
\draw (0.25, 1.15) -- (0.25, .75) -- (.6, .75);
\draw (1.75, .85) -- (1.75, 1.25) -- (1.4, 1.25);
\draw (.75, .6) -- (.75, .25) -- (1.15, .25);
\draw (1.25, 1.4) -- (1.25, 1.75) -- (.85, 1.75);
\filldraw (.6, 1.4) circle (1pt);
\filldraw (.6, 1.3) circle (1pt);
\filldraw (.7, 1.35) circle (1pt);
\filldraw (1.4, .6) circle (1pt);
\filldraw (1.4, .7) circle (1pt);
\filldraw (1.3, .65) circle (1pt);
\draw (.85, 1.1) -- (.85, .85) -- (1.1, .85);
\draw (.9, 1.15) -- (1.15, 1.15) -- (1.15, .9); 
\filldraw (.65, .65) circle (1pt);
\filldraw (.75, .75) circle (1pt);
\filldraw (1.35, 1.35) circle (1pt);
\filldraw (1.25, 1.25) circle (1pt);
\end{tikzpicture}
\]
These double complexes can also be obtained for some complex structures on $\mathfrak{h}_2$ and $\mathfrak{h}_4$. Furthermore, in dimension $6$, these are the only double complexes coming from nilmanifolds that have the same Betti, Hodge and Bott-Chern numbers but different zigzag multiplicities.
\end{Theorem}
%%%%%%%%%%%%%%%%%%%%%%%%%%%%%%%%%%%

Moreover, we revisit the {\sl almost abelian nilmanifolds} considered in \cite{AABRW}, which are nilmanifolds with the Lie algebra having an abelian subalgebra of codimension $1$. Their complex structures are studied in~\cite{AABRW}, and the Betti and Hodge numbers are computed by the means of $\mathfrak{sl}_2(\mathbb{C})$ representation theory. We treat a subclass of almost abelian nilmanifolds for which a certain parameter $k_0 = 0$, and describe the decomposition of their double complexes into irreducible parts.
\begin{Theorem}
The Chevalley-Eilenberg complex of any almost abelian Lie algebra with a complex structure such that $k_0=0$ in its structure equations~\ref{str_eq_ab} is $dd^c+3$, i.e. the only nontrivial zigzags have length 1 or 3. Moreover, there is a representation theoretic formula for all multiplicities of zigzags.
\end{Theorem}

%%%%%%%%%%%%%%%%%%%%%%%%%%%%%%%%%%%%%%%%%%

%%%%%%%%%%%%%%%%%%%%%%%%%%%%%%%%%%%%%%%%%%

\section{General theory about double complexes}
\label{deux}
The aim of this section is to recall the general definition of a double complex, and to explain how they can be decomposed into irreducible parts. In order to gain information about the decomposition of a double complex, various cohomologies and spectral sequences can be considered.
Ultimately, we want to know how much information is necessary to fully recover the decomposition, and hence the double complex (up to isomorphism).
\subsection{Zigzags and Squares}
We shall start by giving the definition of a general double complex :

\begin{Definition}
A double complex is the data of a bigraded vector space :
\[
\mathcal{A}
=
\bigoplus_{i, j \in \mathbb{Z}}
\mathcal{A}^{i, j},
\]

together with two maps $\partial$, $\overline{\partial}$ of bidegree respectively $( + 1, 0)$ and $(0, + 1)$ :
\[
\aligned
\partial
:
\mathcal{A}^{i, j}
\longrightarrow
\mathcal{A}^{i+1, j}, \\
\overline{\partial}
:
\mathcal{A}^{i, j}
\longrightarrow
\mathcal{A}^{i, j + 1},
\endaligned
\]

that satisfy the following relations : 
\[
\aligned
\partial \partial 
&=
0, \\
\overline{\partial} \overline{\partial}
&=
0, \\
\partial \overline{\partial} + \overline{\partial} \partial
&=
0.
\endaligned
\]

Let $\mathcal{A}^{\bullet, \bullet}$ be a double complex. We say that $\mathcal{A}^{p, q}$ is the {\sl component of degree} $(p, q)$. Furthermore, the {\sl component of total degree} $k$ is defined to be
\[
\bigoplus_{p+q=k}
\mathcal{A}^{p, q}.
\]
\end{Definition}
\begin{Definition}
A double complex $(\mathcal{A}^{\bullet, \bullet}, \partial, \overline{\partial})$ is {\sl bounded} if $\mathcal{A}^{p, q} = 0$ for $|p| + |q|$ large enough.
\end{Definition}
We now define the morphisms between double complexes.
\begin{Definition}
Let $(\mathcal{A}^{\bullet, \bullet}, \partial_A, \overline{\partial}_A)$ and $(\mathcal{B}^{\bullet, \bullet}, \partial_B, \overline{\partial}_B)$ be two double complexes. A morphism $f: \mathcal{A} \rightarrow \mathcal{B}$ is a collection of maps
\[
f^{p, q} : \mathcal{A}^{p, q} \longrightarrow \mathcal{B}^{p, q},
\]
such that $f \circ \partial_A = \partial_B \circ f$ and $f \circ \overline{\partial}_A = \overline{\partial}_B \circ f$.
\end{Definition}
The category of double complexes will be denoted $\text{DC}$. The subcategory of bounded double complexes will be denoted $\text{DC}^{\text{bounded}}$.

\begin{Definition}
We define the direct sum of two double complexes $(\mathcal{A}, \partial_{\mathcal{A}}, \overline{\partial}_{\mathcal{A}})$ and $(\mathcal{B}, \partial_{\mathcal{B}}, \overline{\partial}_{\mathcal{B}})$ by taking the direct sum in each degree. The direct sum is therefore the double complex $(\mathcal{C}, \partial_{\mathcal{B}}, \overline{\partial}_{\mathcal{B}})$ such that for every $i, j \in \mathbb{Z}$ :

\[
\aligned
\mathcal{C}^{i, j}
=
\mathcal{A}^{i, j} \oplus \mathcal{B}^{i, j}, \\
\partial_{\mathcal{C}}
=
\partial_{\mathcal{A}} \oplus \partial_{\mathcal{B}}, \\
\overline{\partial}_{\mathcal{C}}
=
\overline{\partial}_{\mathcal{A}} \oplus \overline{\partial}_{\mathcal{B}}.
\endaligned
\]
If a double complex $\mathcal{C}$ can not be written as a direct sum $\mathcal{C} = \mathcal{A} \oplus \mathcal{B}$ of two proper double complexes, then $\mathcal{C}$ is said to be {\sl irreducible}.
\end{Definition}
There are two types of irreducible double complexes. First of all, we have the {\sl square} :

\[
\begin{tikzcd}
\mathbb{C} \arrow[r, "1"] 
& \mathbb{C} \\
\mathbb{C} \arrow[r, "1"] \arrow[u, "1"]
& \mathbb{C} \arrow[u, "-1"] 
\end{tikzcd},
\]
which can have its components in any bidegree.

The second type of irreducible bounded double complexes are the zigzags. We shall separate the zigzags in two classes: the {\sl odd zigzags} and the {\sl even zigzags}. An odd zigzag is a zigzag having an odd total dimension, and having the following form:
\[
\begin{tikzcd}[row sep=small, column sep = small]
\mathbb{C} & & \\
\mathbb{C} \arrow[u] \arrow[r] & \mathbb{C} & \\
& \mathbb{C} \arrow[u] \arrow[r] & \mathbb{C}
\end{tikzcd},
\quad \quad
\text{or} 
\quad \quad
\begin{tikzcd}[row sep=small, column sep = small]
\mathbb{C}\arrow[r] & \mathbb{C} & \\
& \mathbb{C} \arrow[u] \arrow[r] & \mathbb{C} \\
& & \mathbb{C} \arrow[u]
\end{tikzcd}.
\]
As for the square, such a zigzag can have its components in any bidegree.

The {\sl length} of a zigzag is defined as its total dimension. A zigzag of length $1$ is called a {\sl dot}. For an odd zigzag of length at least $2$, there are two consecutive integers $k$ and $k+1$ such that all of its components are of total degree $k$ or $k+1$. If most of the components are in total degree $l$, we say that the {\sl total degree} of the zigzag is $l$. We shall distinguish the cases where $l = k$ and where $l = k+1$:
\begin{Definition}
\label{odd_zigzags}
Let $Z$ be an odd zigzag that is not of length $1$, and with components in total degree $k$ and $k+1$.
\begin{itemize}
\item $Z$ is said to be {\sl looking up} if its total degree is $k+1$.
\item $Z$ is said to be {\sl looking down} if its total degree is $k$.
\end{itemize}
\end{Definition}

Following \cite{Ste}, we shall introduce the notations:

\begin{Notation}
\label{notations_zigzags}
Let $Z$ be an odd zigzag of total degree $k$. Let $(a, b)$ be the bidegree of its top left component, and $(a', b')$ the bidegree of its bottom right component.
\begin{itemize}
\item If the zigzag $Z$ is looking up, we denote it by $Z^{a', b}_{k}$.
\item If the zigzag $Z$ is looking down, we denote it by $Z^{a, b'}_k$.
\item If the zigzag $Z$ is a dot, we denote it by $Z^{a, b}_k$.
\end{itemize}
\end{Notation}

The previous two zigzags both have length $5$. The zigzag on the left is looking up, while the zigzag on the right is looking down.

An even zigzag is a zigzag of even length, and can be either {\sl horizontal} or {\sl vertical}:
\[
\begin{tikzcd}[row sep=small, column sep = small]
\mathbb{C}\arrow[r] & \mathbb{C} & \\
& \mathbb{C} \arrow[u] \arrow[r] & \mathbb{C}
\end{tikzcd},
\quad \quad
\text{or} 
\quad \quad
\begin{tikzcd}[row sep=small, column sep = small]
\mathbb{C} & \\
\mathbb{C} \arrow[u] \arrow[r] & \mathbb{C} \\
& \mathbb{C} \arrow[u]
\end{tikzcd},
\]
where the components can be placed in any bidegree.

Following the notation introduced in \cite{Ste}, we denote a horizontal even zigzag by
\[
Z^{p, q}_{1, l},
\]
where $(p, q)$ is the degree of the most left component of the sigzag, and $2l$ is the length of the zigzag. Similarly, a vertical zigzag is denoted by
\[
Z^{p, q}_{2, l},
\]
where $(p, q)$ is the degree of the bottom component of the zigzag, and $2l$ is the length of the zigzag.
The following Theorem is proven independently by~\cite{Ste} and~\cite{KQ}:
\begin{Theorem}
Every bounded double complex $(\mathcal{A}^{\bullet, \bullet}, \partial, \overline{\partial})$ is isomorphic to a direct sum
\[
\bigoplus_S S^{\text{mult}_S(\mathcal{A})},
\]
where the sum is taken over all squares and zigzags.
\end{Theorem}

\begin{Definition}
A double complex is said to be {\sl minimal} if $\partial \overline{\partial} = 0$. Equivalently, such a double complex has no square in its decomposition.
\end{Definition}

\begin{Remark}
For a double complex that is not necessarily bounded, a similar statement is true, but the direct sum might contain zigzags and squares with components in bidegree $(p, q)$ for arbitrarily large $|p| + |q|$. It may also be the case that the decomposition contains an {\sl infinite zigzag}:
\[
\begin{tikzcd}[row sep=small, column sep = small]
\vdots & & \\
\mathbb{C}\arrow[r] \arrow[u]& \mathbb{C} & \\
& \mathbb{C} \arrow[u] \arrow[r] & \mathbb{C} \\
& & \vdots \arrow[u]
\end{tikzcd}.
\]
\end{Remark}

In the cases in which we are interested, which come from complex geometry, the double complexes are always bounded and in their decomposition only a finite number of irreducible parts are zigzags. We denote the category of bounded double complexes with finite number of zigzags by $\text{DC}^{\text{bounded}, \text{fin}}$. We may consider the free abelian group generated by the isomorphism classes of $\text{DC}^{\text{bounded}, \text{fin}}$, where the sum $[A] + [B]$ is identified with $[A \oplus B]$. The obtained free $\mathbb{Z}$-module will be denoted $\mathcal{U}^{\text{bounded}, \text{fin}}$. Furthermore, let $\mathcal{I}$ be the subgroup of $\mathcal{U}^{\text{bounded}, \text{fin}}$ generated by squares, and denote $\overline{\mathcal{U}^{\text{bounded}, \text{fin}}} = \mathcal{U}^{\text{bounded}, \text{fin}} / \mathcal{I}$. Any element in $\overline{\mathcal{U}^{\text{bounded}, \text{fin}}}$  can be represented by a formal difference of minimal double complexes.

We define now a privileged class of double complex, whose zigzags and squares satisfy some additional properties.
\begin{Definition}
\label{dc-acc}
We define $\text{DC}^{\text{form}}_n$ to be the category of double complexes $\mathcal{A}^{\bullet, \bullet}$ such that:
\begin{itemize}
\item $\mathcal{A}^{p, q} = 0$ if $p < 0$ or $p > n$ or $q < 0$ or $q > n$.
\item There are only dots in the corners $(0, 0)$, $(n, 0)$, $(0, n)$, $(n, n)$.
\item There is a finite number of zigzags.
\item The zigzags are symmetric with respect to the diagonal (going from $(n, 0)$ to $(0, n)$), and also with respect to the anti-diagonal (going from $(0, 0)$ to $(n, n)$).
\end{itemize}
\end{Definition}
The reason we define $\text{DC}^{\text{form}}_n$ in such a way is that the double complex of any compact complex manifold lies in $\text{DC}^{\text{form}}_n$. It is a conjecture that, for $n \geqslant 3$, every double complex in $\text{DC}^{\text{form}}_n$ can be obtained as a formal difference of double complexes coming from compact complex manifolds.

Notice that the direct sum of two double complexes in $\text{DC}^{\text{form}}_n$ is again in $\text{DC}^{\text{form}}_n$. We can therefore consider the free abelian group generated by the isomorphism classes of $\text{DC}^{\text{form}}_n$, and identify the sum $[A] + [B]$ of two classes with the class $[A \oplus B]$. The obtained free $\mathbb{Z}$-module shall be denoted by $\mathcal{U}^{\text{form}}_n$. If $\mathcal{I}$ is the subgroup of $\mathcal{U}^{\text{form}}_n$ generated by squares, then we denote the quotient by $\overline{\mathcal{U}^{\text{form}}_n} = \mathcal{U}^{\text{form}}_n / \mathcal{I}$. This module is generated by the irreducible double complexes satisfying the conditions of Definition~\ref{dc-acc}.

\begin{Property}
\label{comptage}
For all $n \geqslant 2$, the rank of $\overline{\mathcal{U}^{\text{form}}_n}$ is equal to
\[
\text{rank} \ \overline{\mathcal{U}^{\text{form}}_n}
=
\frac{1}{4}
\Big(\frac{4n^3 - 4n}{3}
+5n^2
- 5
+ \delta_{n \in 2\mathbb{Z}}
\Big).
\]
\end{Property}
\begin{proof}
Let $\mathcal{V}$ be the free abelian group generated by the zigzags living in the square $(0, 0) -- (n, 0) -- (n, n) -- (0, n)$, and not touching the corners. The group $\mathbb{Z}/2\mathbb{Z} \times \mathbb{Z}/2\mathbb{Z}$ acts on $\mathcal{V}$ in the following way:
\begin{itemize}
\item $(1, 0)$ is the symmetry with respect to the diagonal.
\item $(0, 1)$ is the symmetry with respect to the anti-diagonal.
\item $(1, 1)$ is the rotation of angle $\pi$ around the center of the square.
\end{itemize}
Now, if $Z$ is some zigzag in $\mathcal{V}$, then $\sum_{g \in \mathbb{Z}/2\mathbb{Z} \times \mathbb{Z}/2\mathbb{Z}} g \cdot Z$ is an element in $\overline{\mathcal{U}^{\text{form}}_n}$. In fact, the family
\[
\Big\{
\sum_{g \in \mathbb{Z}/2\mathbb{Z} \times \mathbb{Z}/2\mathbb{Z}} g \cdot Z
\Big\}_{Z \in \mathcal{V}}
\]
is a basis of $\mathcal{U}^{\text{form}}_n$, and is equinumerous to the number of orbits of the action of $\mathbb{Z}/2\mathbb{Z} \times \mathbb{Z}/2\mathbb{Z}$ on $\mathcal{V}$.

We shall now, for every element $g$ in the group $ \mathbb{Z}/2\mathbb{Z} \times \mathbb{Z}/2\mathbb{Z}$, describe the fixed elements $\text{Fix}(g)$ of $g$ in $\mathcal{V}$. We start with the fixed elements of $(0, 0)$, which is the whole set of zigzags in $\mathcal{V}$. There are $(n+1)^2-4$ dots that can be placed in the square $(0, 0) -- (n, 0) -- (n, n) -- (0, n)$ if we remove the corners. For the zigzags of length at least $2$, consider the following zigzags in the square without the corners:
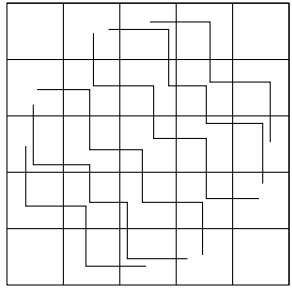
\[
\begin{tikzpicture}[inner sep=0pt, outer sep=0pt, remember picture]
\draw (0, 0) -- (3.75, 0) -- (3.75, 3.75) -- (0, 3.75) -- (0, 0);
\draw (0, .75) -- (3.75, .75);
\draw (0, 1.5) -- (3.75, 1.5);
\draw (0, 2.25) -- (3.75, 2.25);
\draw (0, 3) -- (3.75, 3);
\draw (0, 3.75) -- (3.75, 3.75);
\draw (.75, 0) -- (.75, 3.75);
\draw (1.5, 0) -- (1.5, 3.75);
\draw (2.25, 0) -- (2.25, 3.75);
\draw (3, 0) -- (3, 3.75);
\draw (3.75, 0) -- (3.75, 3.75);
\draw (.25, 1.85) -- (.25, 1.05) -- (1.05, 1.05) -- (1.05, .25) -- (1.85, .25);
\draw (.35, 2.4) -- (.35, 1.6) -- (1.1, 1.6) -- (1.1, 1.1) -- (1.6, 1.1) -- (1.6, .35) -- (2.4, .35);
%%%%%%%%%%%%%%%%%%%%%%%%%%%%%%%%%%%%%%%%%%%%%%%%%%%
\draw (.4, 2.6) -- (1.1, 2.6) -- (1.1, 1.8) -- (1.8, 1.8) -- (1.8, 1.1) -- (2.6, 1.1) -- (2.6, .4);
\draw (3.35, 1.15) -- (2.65, 1.15) -- (2.65, 1.95) -- (1.95, 1.95) -- (1.95, 2.65) -- (1.15, 2.65) -- (1.15, 3.35);
\draw (3.4, 1.35) -- (3.4, 2.15) -- (2.65, 2.15) -- (2.65, 2.65) -- (2.15, 2.65) -- (2.15, 3.4) -- (1.35, 3.4);
\draw (3.5, 1.9) -- (3.5, 2.7) -- (2.7, 2.7) -- (2.7, 3.5) -- (1.9, 3.5);
\end{tikzpicture}
\]
Choosing a zigzag of length $\geq 2$ in $\mathcal{V}$ is equivalent to choosing two points in one of the zigzags in the above picture. Thus, there are
\[
\aligned
\text{Fix} (0, 0)
&=
(n+1)^2-4 
+
\sum_{k = 1}^{n-3}2\binom{2k + 3}{2}
+
4\binom{2n-1}{2} \\
&=
(n+1)^2-4
+
\sum_{k=1}^{n-3}
(2k+3)(2k+2)
+
2(2n-1)(2n-2) \\
&=
\frac{4n^3}{3} + 4n^2 - \frac{13n}{3} - 7
\endaligned
\]
elements in $\mathcal{V}$.

We shall now compute the fixed elements of the symmetry $(0, 1)$. A zigzag is symmetric with respect to the anti-diagonal if and only if it is obtained by choosing two dots on a zigzag in the previous picture, with both dots being symmetric with respect to the anti-diagonal (they may eventually coincide, in which case the obtained zigzag is a dot). Choosing two dots which are symmetric is equivalent to choosing one dot on one side, and we have thus
\[
\aligned
\text{Fix} (0, 1)
&=
n-1
+
2\sum_{k=2}^{n-2} k
+
4(n-1) \\
&=
n^2+2n-5.
\endaligned
\]
For the symmetry along the diagonal, only the dots on the diagonal will be fixed:
\[
\text{Fix} (1, 0)
=
n-1.
\]
Finally, the rotation $(1, 1)$ has no fixed points in odd dimension, and in even dimension it has exactly one fixed point: the dot lying in the middle square.
\[
\text{Fix} (1, 1)
=
\delta_{n \in 2\mathbb{Z}}.
\]
Combining the previous results, and adding $2$ for the dots in the corners, we get:
\[
\text{rank} \ \overline{\mathcal{U}^{\text{form}}_n}
=
\frac{1}{4}
\Big(\frac{4n^3 - 4n}{3}
+5n^2
- 13
+ \delta_{n \in 2\mathbb{Z}}
\Big)
+
2.
\]
\end{proof}
The rank of $\overline{\mathcal{U}^{\text{form}}_1}$ is equal to $2$. The ranks of $\overline{\mathcal{U}^{\text{form}}_n}$ for $n= 2, 3, \dots, 8$ are
\[
6, 18, 39, 70, 114, 172, 247, \dots
\]

%%%%%%%%%%%%%%%%%%%%%%%%%%%%%%%%%
\subsection{The cohomological toolbox}
%%%%%%%%%%%%%%%%%%%%%%%%%%%%%%%%%

Finding the decomposition into zigzags of a double complex demands tedious computations. 
Therefore, we seek to obtain information about the multiplicities of the zigzags by using cohomologies. 
There are different cohomologies that one can consider. If $\mathcal{A}$ is a double complex, then its {\sl total complex} $\text{Tot} \mathcal{A}$ is defined as

\[
\text{Tot} \mathcal{A}^n
=
\bigoplus_{i + j = n}
\mathcal{A}^{i, j},
\quad \quad
\forall n \in \mathbb{Z},
\]
with the differential $d = \partial + \overline{\partial}$. We can check that $d^2 = 0$, so 
$(\text{Tot}\mathcal{A}, d)$
is a complex.

\begin{Definition}
The cohomology of the complex $(\text{Tot}\mathcal{A}, d)$ is called the {\sl de Rham cohomology} and will be denoted $H_{dR}^{\bullet}(\mathcal{A})$ :
\[
H_{dR}^{\bullet}(\mathcal{A})
=
\frac{\ker(d)}{\text{Im}(d)}.
\]
Furthermore, the dimension of $H_{dR}^n(\mathcal{A})$ is called the $n$-th Betti number and will be denoted $b_n(\mathcal{A})$.
\end{Definition}

The Betti numbers count the number of odd zigzags of some total degree.

\begin{Lemma}
Let $\mathcal{A}$ be a double complex. Then, for every integer $n$, the Betti number $b_n(\mathcal{A})$ is equal to the number of odd zigzags of total degree $n$ in the decomposition of $\mathcal{A}$.
\end{Lemma}

\begin{Example}
\label{exdc}
Consider the following double complex, which is already decomposed as a sum of zigzags:
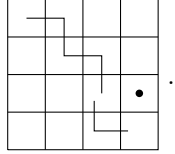
\[
\begin{tikzpicture}[baseline=5.8ex, inner sep=0pt, outer sep=0pt, remember picture]
\draw (0, 0) -- (2, 0) -- (2, 2) -- (0, 2) -- (0, 0);
\draw (.5, 0) -- (.5, 2);
\draw (1, 0) -- (1, 2);
\draw (1.5, 0) -- (1.5, 2);
\draw (0, .5) -- (2, .5);
\draw (0, 1) -- (2, 1);
\draw (0, 1.5) -- (2, 1.5);
\draw (.25, 1.75) -- (.75, 1.75) -- (.75, 1.25) -- (1.25, 1.25) -- (1.25, .75);
\draw (1.15, .65) -- (1.15, .25) -- (1.6, .25);
\fill [black] (1.75, .75) circle [radius=.05];
\end{tikzpicture} \ .
\]
It has one zigzag of length $5$ and one of length $3$, both having total degree $3$ (the bottom left box has bidegree $(0, 0)$). There is also a dot of total degree $4$. Thus the nonzero Betti numbers are $b_3 = 2$ and $b_4 = 1$.
\end{Example}

\begin{Definition}
The cohomology associated to the vertical operator $\overline{\partial}$ is called the {\sl Dolbeault cohomology} and will be denoted $H_{\overline{\partial}}^{\bullet, \bullet}$ :
\[
H_{\overline{\partial}}^{\bullet, \bullet} 
=
\frac{\ker \overline{\partial}}{\text{Im} \overline{\partial}}.
\]
The dimensions of these vector spaces are called the {\sl Hodge numbers} and will be denoted by
\[
h_{\overline{\partial}}^{\bullet, \bullet}
=
\dim H_{\overline{\partial}}^{\bullet, \bullet} .
\]
\end{Definition}

As for the de Rham cohomology, the Dolbeault cohomology vanishes for squares. For a zigzag, the only parts that do not vanish in Dolbeault cohomology are the ones whose vertical arrows are zeros.

\[
\begin{tikzcd}
\mathbb{C} \arrow[r] & \mathbb{C}\\
&\mathbb{C} \arrow[r] \arrow[u] & \mathbb{\ddots} \\
 & &\mathbb{C} \arrow[r] \arrow[u] & \mathbb{C} 
\end{tikzcd}
\substack{H_{\overline{\partial}} \\ \Longrightarrow}
\begin{tikzcd}
\mathbb{C} \arrow[r] & 0\\
& 0 \arrow[r] \arrow[u] & \mathbb{\ddots} \\
 & & 0 \arrow[r] \arrow[u] & \mathbb{C} 
\end{tikzcd}
\]

In Example~\ref{exdc}, the nonzero Hodge numbers are
$h^{0, 3}_{\overline{\partial}}
=
h^{3, 0}_{\overline{\partial}}
=
h^{3, 1}_{\overline{\partial}}
=
1$.

Furthermore, we can define, for any double complex, the {\sl Bott-Chern cohomology groups}:
\[
H_{\text{BC}}^{\bullet, \bullet}
=
\frac{\ker \partial \cap \ker \overline{\partial}}{\text{Im}\partial \overline{\partial}}.
\]
The Bott-Chern cohomology vanishes on squares. For a zigzag, it identifies the degrees which have no outgoing arrows. For different types of zigzags, the red dots in the following picture show the degrees that contribute in the Bott-Chern cohomology:

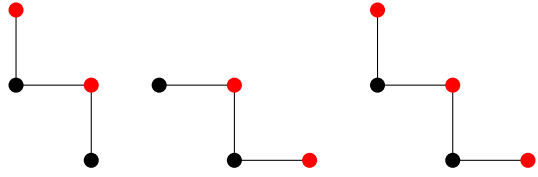
\[
\begin{tikzpicture}
\draw (1, -1) -- (1, 0) -- (0, 0) -- (0, 1);
\fill [red] (1,0) circle [radius=.1];
\fill [black] (1,-1) circle [radius=.1];
\fill [black] (0,0) circle [radius=.1];
\fill [red] (0,1) circle [radius=.1];
\end{tikzpicture}
\quad \quad
\begin{tikzpicture}
\draw (1, -.5) -- (0, -.5) -- (0, .5) -- (-1, .5);
\fill [red] (1,-.5) circle [radius=.1];
\fill [black] (0,-.5) circle [radius=.1];
\fill [black] (-1,.5) circle [radius=.1];
\fill [red] (0,.5) circle [radius=.1];
\end{tikzpicture}
\quad \quad
\begin{tikzpicture}
\draw (1, -.5) -- (0, -.5) -- (0, .5) -- (-1, .5) -- (-1, 1.5);
\fill [red] (1,-.5) circle [radius=.1];
\fill [black] (0,-.5) circle [radius=.1];
\fill [black] (-1,.5) circle [radius=.1];
\fill [red] (0,.5) circle [radius=.1];
\fill [red] (-1,1.5) circle [radius=.1];
\end{tikzpicture}
\]

In Example~\ref{exdc}, the nonzero Bott-Chern numbers are
$h^{1, 3}_{\text{BC}}
=
h^{2, 2}_{\text{BC}}
=
h^{2, 1}_{\text{BC}}
=
h^{3, 1}_{\text{BC}}
=
h^{3, 0}_{\text{BC}}
=
1$.

In a similar way we introduce the {\sl Aeppli cohomology groups}:
\[
H_{\text{A}}^{\bullet, \bullet}
=
\frac{\ker \partial \overline{\partial}}{\text{Im} \partial + \text{Im} \overline{\partial}}.
\]
The Aeppli cohomology vanishes on squares. For a zigzags, it identifies the degrees which have only outgoing arrows. For different types of zigzags, the red dots in the following picture show the degrees that contribute in the Aeppli cohomology:
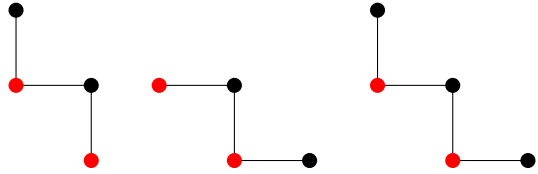
\[
\begin{tikzpicture}
\draw (1, -1) -- (1, 0) -- (0, 0) -- (0, 1);
\fill [black] (1,0) circle [radius=.1];
\fill [red] (1,-1) circle [radius=.1];
\fill [red] (0,0) circle [radius=.1];
\fill [black] (0,1) circle [radius=.1];
\end{tikzpicture}
\quad \quad
\begin{tikzpicture}
\draw (1, -.5) -- (0, -.5) -- (0, .5) -- (-1, .5);
\fill [black] (1,-.5) circle [radius=.1];
\fill [red] (0,-.5) circle [radius=.1];
\fill [red] (-1,.5) circle [radius=.1];
\fill [black] (0,.5) circle [radius=.1];
\end{tikzpicture}
\quad \quad
\begin{tikzpicture}
\draw (1, -.5) -- (0, -.5) -- (0, .5) -- (-1, .5) -- (-1, 1.5);
\fill [black] (1,-.5) circle [radius=.1];
\fill [red] (0,-.5) circle [radius=.1];
\fill [red] (-1,.5) circle [radius=.1];
\fill [black] (0,.5) circle [radius=.1];
\fill [black] (-1,1.5) circle [radius=.1];
\end{tikzpicture}
\]

In Example~\ref{exdc}, the nonzero Aeppli numbers are
$h^{0, 3}_{\text{A}}
=
h^{1, 2}_{\text{A}}
=
h^{2, 1}_{\text{A}}
=
h^{2, 0}_{\text{A}}
=
h^{3, 0}_{\text{A}}
=
1$.

\begin{Remark}
Let $Z$ be a zigzag, and $\star Z$ its symmetry wih respect to the diagonal going from $(0, n)$ to $(n, 0)$. Then the components of $Z$ having ingoing arrows correspond to the components of $\star Z$ having outgoing arrows. Hence, the information of the Bott-Chern cohomology of $Z$ is equivalent to the Aeppli cohomology of $\star Z$. 

In particular, for double complexes in $\mathcal{U}_n^{\text{form}}$, which are symmetric along the diagonal, knowing the Bott-Chern cohomology is equivalent to knowing the Aeppli cohomology.
\end{Remark}

%%%%%%%%%%%%%%%%%%%%%%%%%%%%%%%%%%%%%%%%%%%
\subsection{The Frölicher spectral sequence}
%%%%%%%%%%%%%%%%%%%%%%%%%%%%%%%%%%%%%%%%%%%
Let $(K^{\bullet, \bullet}, \partial, \overline{\partial})$ be a double complex. We will describe the {\sl Frölicher spectral sequence} of $K$.

The zero page of the spectral sequence is simply the double complex $K$ with the differential $\overline{\partial}$:
\[
(E_0^{\bullet, \bullet}, d_0)
=
(K^{\bullet, \bullet}, \overline{\partial}).
\]
Thus, the first page is given by the Dolbeault cohomologies, and the differential is the map induced by $\partial$ in the quotient:
\[
(E_1^{\bullet, \bullet}, d_1)
=
(H^{\bullet, \bullet}_{\overline{\partial}}(K), [\partial]).
\]
We shall describe the Frölicher spectral sequence for the different irreducible double complexes. First of all, any square has vanishing Dolbeault cohomology and thus the Frölicher spectral sequence of a square vanishes (and degenerates) at the first page.

Now, for an odd zigzag, the Dolbeault cohomology vanishes in every but one bidegree. Therefore the first page of the spectral sequence of an odd zigzag has only one nonzero component. All the maps are trivial, and thus the Frölicher spectral sequence degenerates at first page.

For an even vertical zigzag, the Dolbeault cohomology vanishes everywhere and so does the first page $E_1$. For an even horizontal zigzag $Z^{a, b}_{1, l}$, the two extremities of bidegree $(a, b)$ and $(a+l, b-l+1)$ will have non vanishing Dolbeault cohomology. Now the arrows in the pages $E_1, E_2, \dots, E_{l-1}$ are all zero, so we have
\[
E_1 \simeq E_2 \simeq \cdots \simeq E_{l}.
\]
In $E_l$ however, the arrow starting from bidegree $(a, b)$ goes to bidegree $(a+l, b-l+1)$. Since the Frölicher spectral sequence converges to the De Rham cohomology, which vanishes for even zigzags, this arrow must be an isomorphism. 
Therefore, the spectral sequence vanishes (and degenerates) at page $l+1$.
\[
\aligned
E_0 :
\begin{tikzcd}[row sep=.8em, column sep=.8em, baseline=-1em]
\vdots & \vdots & \vdots & \vdots \\
\mathbb{C} \arrow[u] & \mathbb{C} \arrow[u] & 0 \arrow[u] & 0 \arrow[u]\\
0 \arrow[u] &\mathbb{C} \arrow[u] & \mathbb{C} \arrow[u] & 0 \arrow[u] \\
0 \arrow[u]  & 0 \arrow[u] &\mathbb{C} \arrow[u] & \mathbb{C} \arrow[u] 
\end{tikzcd}
\quad \quad
E_1:
\begin{tikzcd}[row sep=.8em, column sep=.8em]
\mathbb{C} \arrow[r] & 0 \arrow[r] & 0 \arrow[r] & 0 \arrow[r] & \cdots \\
0 \arrow[r] &  0\arrow[r] & 0 \arrow[r] & 0 \arrow[r] & \cdots \\
0 \arrow[r]  & 0 \arrow[r] & 0 \arrow[r] & \mathbb{C} \arrow[r] & \cdots
\end{tikzcd}
\\ \\
E_2:
\begin{tikzcd}[row sep=.8em, column sep=.8em]
\mathbb{C} \arrow[rrd] & 0 \arrow[rrd] & 0 \arrow[rrd] & 0 & \cdots \\
0 \arrow[rrd] &  0\arrow[rrd] & 0 \arrow[rrd] & 0 & \cdots \\
0 & 0 & 0 & \mathbb{C} & \cdots
\end{tikzcd}
\quad \quad
E_3:
\begin{tikzcd}[row sep=.8em, column sep=.8em]
\mathbb{C} \arrow[rrrdd] & 0 \arrow[rrrdd] & 0 & 0 & \cdots \\
0 &  0 & 0 & 0 & \cdots \\
0 & 0 & 0 & \mathbb{C} & \cdots
\end{tikzcd}
\endaligned
\]
The Frölicher spectral sequence therefore establishes a relation between the Betti numbers (which are topological invariants) and the Hodge numbers. In terms of zigzags, the Frölicher spectral sequence gives the multiplicities of the even horizontal zigzags. 
\begin{Fact}
The data of the Frölicher spectral sequence is enough to determine the multiplicities of even horizontal zigzags in the decomposition of a double complex.
\end{Fact}
The Frölicher spectral sequence is the spectral sequence associated to the following filtration of the total complex:
\[
F_l \text{Tot}^k K
=
\bigoplus_{
\substack{
p+q=k \\
l \leqslant p}
}
K^{p, q}.
\]
One can consider a conjugate filtration:
\[
\tilde{F_l} \text{Tot}^k K
=
\bigoplus_{
\substack{
p+q=k \\
l \leqslant q}
}
K^{p, q},
\]
which gives rise to the {\sl conjugate Frölicher spectral sequence} $\tilde{E}^{\bullet, \bullet}_{\bullet}$. This spectral sequence determines the multiplicities of the even vertical zigzags.

We will simply write the Frölicher spectral sequences to allude to the Frölicher spectral sequence and its conjugate.
\begin{Definition}
We say that a map $f: \mathcal{A} \longrightarrow \mathcal{B}$ is a $E_r$-isomorphism if it induces an isomorphism on the $r$-th page of the Frölicher spectral sequences. 
\end{Definition}
By definition, a $E_1$-isomorphism is a quasi-isomorphism in Dolbeault and conjugate Dolbeault cohomologies. For bounded double complexes, such a map induces an isomorphism in all cohomological invariants (\cite{Ste}). Therefore, a $E_1$-isomorphism will also be called a {\sl quasi-isomorphism}. Two double complexes are quasi-isomorphic if and only if they have the same multiplicities for all zigzags in their decompositions.
%%%%%%%%%%%%%%%%%%%%%%%%%%%%%%%%%%%%%%%%%%%
\section{Computing the multiplicities of zigzags}
%%%%%%%%%%%%%%%%%%%%%%%%%%%%%%%%%%%%%%%%%%%
In this section we are interested in finding the decomposition of double complexes in zigzags. We will consider the {\sl Varouchas cohomologies}, which in dimension $n$ give us $O(n^2)$ linear equations in the zigzags. From Property~\ref{comptage}, we already know that this will not suffice, and we will investigate what type of double complexes are indistinguishable by these cohomologies, together with Betti, Hodge, Bott-Chern and Aeppli numbers. These double complexes will be referred to as {\sl locally similar}, and we shall describe a way of constructing them.

We then consider the {\sl Bigolin cohomology}, which gives $O(n^3)$ linear equations in the zigzags. We describe which zigzags are counted by each Bigolin number, and ultimately prove that the multiplicities of odd zigzags can be computed by them.

%%%%%%%%%%%%%%%%%%%%%%%%%%%%%%%%%%%%%%%%%%%
\subsection{Locally similar double complexes}
%%%%%%%%%%%%%%%%%%%%%%%%%%%%%%%%%%%%%%%%%%%
\label{loc_sim}

One can imagine a double complex as a puzzle with seven distinct pieces :
\begin{equation}
\label{pieces_puzzle}
\case{0}{0} \ ,
\ \
\case{1}{0} \ ,
\ \ 
\case{2}{0} \ ,
\ \
\case{3}{0} \ ,
\ \
\case{4}{0} \ ,
\ \
\case{5}{0} \ ,
\ \
\case{6}{0}
\ .
\end{equation}
One shall imagine these pieces to be transparent, so that they can be stacked to form as many dots and zigzags in every degree as needed.

In each degree, the Bott-Chern cohomology gives us information about how the double complex looks like in this degree. More precisely, the four first pieces displayed above will contribute to the dimension of $H_{\text{BC}}$. We write this schematically as :

\[
h_{\text{BC}}
=
\case{0}{0}
\ + \
\case{1}{0}
\ + \
\case{2}{0}
\ + \
\case{3}{0}
\ .
\]

Similarly, the first piece and the last three will contribute to the dimension of the Aeppli cohomology :

\begin{equation}
\label{eq_A}
h_{\text{A}}
=
\case{0}{0}
\ + \
\case{4}{0}
\ + \
\case{5}{0}
\ + \
\case{6}{0}
\ .
\end{equation}

Notice that a similar expression can not be written for the De Rham cohomolgy, 
as the computation for this cohomology needs the information of the total complex.

\begin{Definition}
Let $(\mathcal{A}^{\bullet, \bullet}, \partial_1, \overline{\partial}_1)$ be a double complex. We define the {\sl local complex} of $\mathcal{A}$ centered in bidegree $(p, q)$ to be the complex $(\mathcal{A}_{p, q}^{\bullet, \bullet}, \partial_2, \overline{\partial}_2)$ defined by
\[
\mathcal{A}_{p, q}^{r, s}
=
\Bigg\{
\begin{array}{ccc}
\mathcal{A}^{r, s} & \text{if } |r-p| + |s - q| \leqslant 1. \\
0 & \text{otherwise}.
\end{array}
\]
The differentials $\partial_2, \overline{\partial}_2$ are defined to be equal to the differentials $\partial_1, \overline{\partial}_1$ when possible, and to be the zero map otherwise:
\[
\begin{array}{cccc}
\partial_2 = \partial_1 : A^{r, s} \mapsto A^{r+1, s} \ \text{if } r =p, p-1,
&
\overline{\partial}_2 = \overline{\partial}_1 : A^{r, s} \mapsto A^{r, s+1} \ \text{if } s =q, q-1.
\\
\partial_2 = 0 : A^{r, s} \mapsto A^{r+1, s} \ \text{if } r \neq p, p-1,
&
\overline{\partial}_2 = 0 : A^{r, s} \mapsto A^{r, s+1} \ \text{if } s \neq q, q-1.
\end{array}
\]
\end{Definition}
\begin{Definition}
Two double complexes $\mathcal{A}^{\bullet, \bullet}$ and $\mathcal{B}^{\bullet, \bullet}$ are said to be {\sl locally similar} if there are two minimal double complexes $\mathcal{A}'^{\bullet, \bullet}$ and $\mathcal{B}'^{\bullet, \bullet}$ such that $\mathcal{A}^{\bullet, \bullet}$ is quasi-isomorphic to $\mathcal{A}'^{\bullet, \bullet}$, $\mathcal{B}^{\bullet, \bullet}$ is quasi-isomorphic to $\mathcal{B}'^{\bullet, \bullet}$, and for all bidegrees $(p, q)$, we have an isomorphism
\[
\mathcal{A}'^{\bullet, \bullet}_{p, q}
\simeq
\mathcal{B}'^{\bullet, \bullet}_{p, q}.
\]
\end{Definition}

\subsection{Varouchas cohomologies}

Following \cite{JV}, we consider for any double complex the six cohomologies :

\[
\begin{array}{cc}
A^{\bullet, \bullet} = \frac{\text{Im} \partial \cap \text{Im} \overline{\partial}}{\text{Im} \partial \overline{\partial}} 
&
D^{\bullet, \bullet}= \frac{\ker \partial \cap \text{Im} \overline{\partial}}{\text{Im} \partial \overline{\partial}}
\\
B^{\bullet, \bullet} = \frac{\text{Im} \partial \cap \ker \overline{\partial}}{\text{Im} \partial \overline{\partial}}
&
E^{\bullet, \bullet} =  \frac{\ker \partial \overline{\partial}}{\ker \partial + \text{Im} \overline{\partial}}
\\
C^{\bullet, \bullet} = \frac{\ker \partial \overline{\partial}}{\ker \overline{\partial} + \text{Im} \partial}
&
F^{\bullet, \bullet} = \frac{\ker \partial \overline{\partial}}{\ker \overline{\partial} + \ker \partial}
\end{array}
\]

Each one of them gives us a linear equation involving the seven pieces displayed in \eqref{pieces_puzzle}. Together with Dolbeault, Bott-Chern and Aeppli, we have a total of $9$ equations, for only $7$ variables. We shall write down the linear system. Each row corresponds to a cohomology, and each column to a puzzle piece, in the same order as in \eqref{pieces_puzzle} :

\[
\begin{pmatrix}
1 & 1 & 1 & 1 & 0 & 0 & 0 \\
1 & 0 & 0 & 0 & 1 & 1 & 1 \\
1 & 0 & 0 & 1 & 0 & 1 & 0 \\
0 & 1 & 0 & 0 & 0 & 0 & 0 \\
0 & 1 & 0 & 1 & 0 & 0 & 0 \\
0 & 0 & 0 & 0 & 1 & 0 & 1 \\
0 & 1 & 1 & 0 & 0 & 0 & 0 \\
0 & 0 & 0 & 0 & 0 & 1 & 1 \\
0 & 0 & 0 & 0 & 0 & 0 & 1
\end{pmatrix} .
\]

A straightforward computation using the help of SAGE shows that this system can be solved, i.e the matrix is injective. We therefore have the
\begin{Property}
Two double complexes $\mathcal{A}$ and $\mathcal{B}$ are locally similar if and only if they share the same Varouchas, Dolbeault, Bott-Chern and Aeppli Cohomologies.
\end{Property}

\subsection{Construction of locally similar double complexes}
For a bounded double complex in $\text{DC}^{\text{form}}_n$, the Varouchas, as well as the Bott-Chern, Aeppli and Hodge cohomologies can each generate at most $n^2$ linear equations on the multiplicities of the zigzags of the double complex. However, by Property~\ref{comptage}, we know that the number of indeterminate zigzags has a cubic growth rate. Therefore, for $n$ sufficiently large, there will be two double complexes that will have the same previously described cohomologies. Equivalently, these double complexes will be locally similar.
 \begin{Fact}
Locally similar, but non quasi-isomorphic, double complexes exist.
\end{Fact}

We are now interested in the construction of locally similar double complexes that have the same Betti numbers and Frölicher spectral sequence. To do that, consider the linear map
\[
\chi
:
\overline{\mathcal{U}^{\text{bounded, fin}}}
\longrightarrow
\mathbb{Z}^{\mathbb{N}},
\]
which sends a formal double complex to the array $(b^{\bullet}, h_{\text{BC}}^{\bullet, \bullet}, h_{\text{A}}^{\bullet, \bullet}, h_{\overline{\partial}}^{\bullet, \bullet}, a^{\bullet, \bullet}, \cdots, f^{\bullet, \bullet}, \mathcal{F}_{\bullet}^{\bullet, \bullet}, \overline{\mathcal{F}_{\bullet}^{\bullet, \bullet}})$ of its Betti, Bott-Chern, Aeppli, Hodge, Varouchas numbers and to the dimensions of the Frölicher spectral sequences $\mathcal{F}_{\bullet}^{\bullet, \bullet}$ and $\overline{\mathcal{F}_{\bullet}^{\bullet, \bullet}}$.

We will give a construction of elements in the kernel of $\chi$ and show that we have described all of them. We start by choosing a degree $k$, and we consider the sets $S^+_k$ and $S^-_k$ of bidegree $(p, q)$ of total degree $p+q > k$ and $p+q < k$ respectively. These sets can be represented with squares placed adequately, for example $S^-_3$:
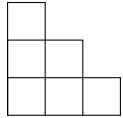
\[
\begin{tikzpicture}[baseline=5.8ex, inner sep=0pt, outer sep=0pt, remember picture]
\draw (0, 0) -- (1.5, 0) -- (1.5, .5) -- (0, .5) -- (0, 0);
\draw (1, 0) -- (1, 1) -- (0, 1) -- (0, 0);
\draw (.5, 0) -- (.5, 1.5) -- (0, 1.5) -- (0, 0);
\end{tikzpicture}
\]
Depending on the context, one can consider only bidegrees in the first quadrant, or also negative bidegrees.

Now, fill the square in the diagram with integers such that:
\begin{itemize}
\item On any row or column, almost all integers are $0$.
\item On any row or column, the (finite) sum of the integers is equal to $0$.
\end{itemize}
\begin{Example}
\label{fill_ex}
In the above diagram, a way to do it is the following:
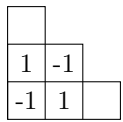
\[
\begin{tikzpicture}[baseline=5.8ex, inner sep=0pt, outer sep=0pt, remember picture]
\draw (0, 0) -- (1.5, 0) -- (1.5, .5) -- (0, .5) -- (0, 0);
\draw (1, 0) -- (1, 1) -- (0, 1) -- (0, 0);
\draw (.5, 0) -- (.5, 1.5) -- (0, 1.5) -- (0, 0);
\node at (0.75, .25) (a) {1};
\node at (0.25, .75) (a) {1};
\node at (0.75, .75) (a) {-1};
\node at (0.25, .25) (a) {-1};
\end{tikzpicture}
\]
\end{Example}
This is equivalent to choosing a mapping $v: S^-_k \rightarrow \mathbb{Z}$. such that
\begin{equation}
\label{condition_v}
\begin{array}{cccc}
v((p, q)) &= 0 \ &\text{for all except a finite number of } (p, q) \in S_k^- \\
\sum_{p \leqslant k-q-1}v((p, q)) &= 0 \ &\text{for all } q \in \mathbb{Z} \\
\sum_{q \leqslant k-p-1}v((p, q)) &= 0 \ &\text{for all } p \in \mathbb{Z} 
\end{array}
\end{equation}

We define the following double complex:
\[
Z^{v-}_k
=
\bigoplus_{(p, q) \in S^-_k} v((p, q)) Z^{p, q}_k.
\]
Similarly, we can do the same thing with the set $S^+_k$:
\[
Z^{v+}_k
=
\bigoplus_{(p, q) \in S^+_k} v((p, q)) Z^{p, q}_k,
\]
where the map $v:S^+_k \longrightarrow \mathbb{Z}$ satisfy
\begin{equation}
\label{condition_v}
\begin{array}{cccc}
v((p, q)) &= 0 \ &\text{for all except a finite number of } (p, q) \in S_k^- \\
\sum_{p \geqslant k-q+1}v((p, q)) &= 0 \ &\text{for all } q \in \mathbb{Z} \\
\sum_{q \geqslant k-p+1}v((p, q)) &= 0 \ &\text{for all } p \in \mathbb{Z} 
\end{array}
\end{equation}
\begin{Property}
For all degree $k \in \mathbb{Z}$ and map $v: S_k^{\pm} \rightarrow \mathbb{Z}$, the double complex $Z^{v\pm}_k$ lies in the kernel of $\chi$.
\end{Property}
\begin{proof}
All the zigzags are odd and of total degree $k$, thus the Betti numbers $b_l$ are all zero for $l \neq k$. Furthermore, $b_k$ is equal to the number of zigzags in the double complex (counted with signs and multiplicities), which is zero by the construction of $v$.

We will now consider the case of the plus sign and show that $Z^{v+}_k$ is in the kernel of $\chi$. The case of the minus sign is covered by symmetry. Since all the zigzags in $Z^{v+}_k$ are looking up, and since we do not have any dots, there are only four puzzle pieces that need to be considered:
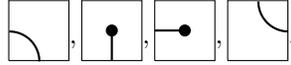
\[
\case{1}{0},
\case{2}{0},
\case{3}{0},
\case{6}{0}.
\] 
We need to show that each one of them appears $0$ times (when counted with multiplicities) in the double complex $Z^{v+}_k$. Once again, by symmetry, we only need to show it for the first two.

Let $(a, b)$ be a bidegree. The number of $\case{2}{0}$ in bidegree $(a, b)$ is zero if $a+b\neq k$. If $a+b=k$, then it is equal to the number of zigzags $Z^{p, q}_k$ with $p > a$ and $q = b$:
\[
\aligned
\case{2}{0}^{a, b}
&=
v((a+1, b)) + v((a+2, b)) + \cdots \\
&=
0,
\endaligned
\]
which vanishes by the definition of $v$.
Consider now the number of $\case{1}{0}$ in bidegree $(a, b)$. If $a+b\neq k$, then this number is again zero. If $a+b=k$, then it is equal to the number of zigzags $Z^{p, q}_k$ with $p > a$ and $q > b$:
\[
\aligned
\case{1}{0}^{a, b}
&=
\sum_{t > 0, s > 0}
v((a+t, b+s)) \\
&=
0.
\endaligned
\]
We deduce that the complex vanishes for the Dolbeault, Bott-Chern, Aeppli and Varouchas cohomologies. Furthermore, the Frölicher spectral sequence degenerates at the first page because there are only odd zigzags. Since the first page is simply the Dolbeault cohomology, it also vanishes.
\end{proof}
The previous construction gives us the family $\{Z_k^{v\pm}\}$ of elements in the kernel of $\chi$. We will now show that this family is in fact a generating family of the kernel.
\begin{Theorem}
The family of elements:
\[
\{Z^{v\pm}_k\},
\]
generates $\ker \chi$.
\end{Theorem}
\begin{proof}
Let $\mathcal{C}$ be an element of $\ker \chi$. Since it has trivial Frölicher spectral sequences, it has no even zigzags, and we can write it as
\[
\mathcal{C}
=
\mathcal{C}_{\text{dots}} \oplus
\bigoplus_{k \in \mathbb{Z}}
\mathcal{C}_k^+ \oplus \mathcal{C}_k^-,
\]
where $\mathcal{C}_k^+$ consists in the odd zigzags of total degree $k$ that are looking up, $\mathcal{C}_k^-$ consists in the odd zigzags of total degree $k$ that are looking down and $\mathcal{C}_{\text{dots}}$ consists in the dots of $\mathcal{C}$. Since the dots can be recovered from the Varouchas cohomologies, we have $\mathcal{C}_{\text{dots}} = 0$. Consider now the zigzags in degree $k$ that are looking down, and write them as
\[
\mathcal{C}_k^-
=
\bigoplus_{p+q<k}v((p, q))Z_k^{p, q},
\]
where $v((p,q)) \in \mathbb{Z}$ are integers. We are done if we prove that $v((p,q))$ satisfy the conditions of~\ref{condition_v}. The fact that $v((p, q))=0$ for all but a finite number of bidegrees $(p, q)$ is a consequence of the fact that $\mathcal{C}$ is an element in $\overline{\mathcal{U}^{\text{bounded, fin}}}$. Furthermore, we have for all bidegrees $(r, s)$ such that $r+s=k$:
\[
\aligned
0
&=
\case{4}{0}^{r, s}
&=
v((r-1, s)) + v((r-2, s)) + \cdots, \\
0
&=
\case{5}{0}^{r, s}
&=
v((r, s-1)) + v((r, s-2)) + \cdots.
\endaligned
\]
We deduce that $v: S_k^- \longrightarrow \mathbb{Z}$ satisfy the conditions of~\ref{condition_v} and $\mathcal{C}_k^- = Z_k^{v-}$. The same is true for the odd zigzags that are looking up $\mathcal{C}_k^+$.
\end{proof}

\begin{Example}
Example~\ref{fill_ex} yields, after separating the negative zigzags from the positive ones, the following two double complexes:
\[
\begin{tikzpicture}[baseline=5.8ex, inner sep=0pt, outer sep=0pt, remember picture]
\draw (0, 0) -- (2, 0) -- (2, 2) -- (0, 2) -- (0, 0);
\draw (.5, 0) -- (.5, 2);
\draw (1, 0) -- (1, 2);
\draw (1.5, 0) -- (1.5, 2);
\draw (0, .5) -- (2, .5);
\draw (0, 1) -- (2, 1);
\draw (0, 1.5) -- (2, 1.5);
\draw (.25, 1.75) -- (.75, 1.75) -- (.75, 1.25) -- (1.25, 1.25) -- (1.25, .75);
\draw (.8, 1.15) -- (1.15, 1.15) -- (1.15, .65) -- (1.6, .65) -- (1.6, .25);
\end{tikzpicture}
\simeq
\begin{tikzpicture}[baseline=5.8ex, inner sep=0pt, outer sep=0pt, remember picture]
\draw (0, 0) -- (2, 0) -- (2, 2) -- (0, 2) -- (0, 0);
\draw (.5, 0) -- (.5, 2);
\draw (1, 0) -- (1, 2);
\draw (1.5, 0) -- (1.5, 2);
\draw (0, .5) -- (2, .5);
\draw (0, 1) -- (2, 1);
\draw (0, 1.5) -- (2, 1.5);
\draw (.25, 1.75) -- (.75, 1.75) -- (.75, 1.25) -- (1.25, 1.25) -- (1.25, .75) -- (1.75, .75) -- (1.75, .25);
\draw (.8, 1.1) -- (1.1, 1.1) -- (1.1, .8);
\end{tikzpicture}
\]
\end{Example}
%%%%%%%%%%%%%%%%%%%%%%%%%%%%%%

%%%%%%%%%%%%%%%%%%%%%%%%%%%%%%%%%%%%%%%%%%%%%
\subsection{The Bigolin complex}
%%%%%%%%%%%%%%%%%%%%%%%%%%%%%%%%%%%%%%%%%%%
\label{bigolin}
Let $(A^{\bullet, \bullet}, \partial, \overline{\partial})$ be a double complex. We denote by $\text{Tot}(A)^{\bullet} = \oplus_{p+q = \bullet} A^{p, q}$ its total complex, with the differential $d = \partial + \overline{\partial}$.

We fix a couple $(p, q)$. For all integer $k$, we denote
\[
\aligned
\mathcal{B}^k_{p, q}(A)
&=
\bigoplus_{
\substack{
r+s=k \\
0 \leqslant r \leqslant p \\
0 \leqslant s \leqslant q}
}
A^{r, s}
\quad
\text{if}
\ \ 
k \leqslant p+q, \\
\mathcal{B}^k_{p, q}(A)
&=
\bigoplus_{
\substack{
r+s=k+1 \\
r > p \\
s > q}
}
A^{r, s}
\quad
\text{if}
\ \ 
k > p+q. \\
\endaligned
\]
We shall also denote by $\pi^{p, q}$ the projection map onto the sum 
$\bigoplus_
{\substack{
0 \leqslant r < p\\
0 \leqslant s < q}}
A^{r, s}$,
sending a differential form on the sum of its $(r, s)$-components, with $0 \leqslant r < p$ and $0 \leqslant s < q$. We are now able to define the {\sl Bigolin complex}:
\[
\mathcal{B}_{p, q}^0
\xrightarrow{\pi^{p, q} \circ d}
\mathcal{B}_{p, q}^1
\xrightarrow{\pi^{p, q} \circ d}
\cdots
\mathcal{B}_{p, q}^{p+q}
\xrightarrow{\partial \overline{\partial}}
\mathcal{B}_{p, q}^{p+q+1}
\xrightarrow{d}
\mathcal{B}_{p, q}^{p+q+2}
\xrightarrow{d}
\cdots.
\]
The associated cohomology is called the {\sl Bigolin cohomology} and will be denoted by $H^k_{p, q}$. The Bigolin cohomology is oblivious to squares. For compact complex manifolds, the number of zigzags is finite by the Hodge Theory, so we deduce that the Bigolin cohomologies are finite dimensional.
We define the {\sl Bigolin numbers} $h^k_{p, q}$ to be the dimensions of the corresponding cohomology groups. The computation of the Bigolin complex can be done by restricting the original double complex to the bidegree in which we are interested. We will use the following formal construction:

\begin{Definition}
Let $\mathcal{A}^{\bullet, \bullet}$ be a double complex. We define the {\sl cutting of } $\mathcal{A}^{\bullet, \bullet}$ {\sl from the above at height} $p$ to be the double complex $\mathcal{A}'$ defined by
\[
\mathcal{A}'^{r, s}
=
\Bigg\{
\begin{array}{ccc}
\mathcal{A}^{r, s} \quad &\text{if } s \leqslant p, \\
0 \quad &\text{if } s > p.
\end{array}.
\]
The {\sl cutting of  }$\mathcal{A}^{\bullet, \bullet}$ {\sl from below at degree} $p$ is defined to be
\[
\mathcal{A}'^{r, s}
=
\Bigg\{
\begin{array}{ccc}
0 \quad &\text{if } s \leqslant p, \\
\mathcal{A}^{r, s} \quad &\text{if } s > p.
\end{array}.
\]
Similarly, the {\sl cutting of }$\mathcal{A}^{\bullet, \bullet}$ {\sl from the right at degree} $p$ is defined to be
\[
\mathcal{A}'^{r, s}
=
\Bigg\{
\begin{array}{ccc}
\mathcal{A}^{r, s} \quad &\text{if } r \leqslant p, \\
0 \quad &\text{if } r > p.
\end{array}.
\]
The {\sl cutting of  }$\mathcal{A}^{\bullet, \bullet}$ {\sl from the left at degree} $p$ is defined to be
\[
\mathcal{A}'^{r, s}
=
\Bigg\{
\begin{array}{ccc}
0 \quad &\text{if } r \leqslant p, \\
\mathcal{A}^{r, s} \quad &\text{if } r > p.
\end{array}.
\]
In each case, the differential $\partial'$ in bidegree $(r, s)$ of $\mathcal{A}'$:
\[
\partial_{r, s}'
:
\mathcal{A}'^{r, s} \longrightarrow \mathcal{A}'^{r+1, s},
\]
is defined to be equal to the differential $\partial_{r, s}$ of $\mathcal{A}$ if $\mathcal{A}'^{r, s} = \mathcal{A}^{r, s}$ and $\mathcal{A}'^{r+1, s} = \mathcal{A}^{r+1, s}$, and to be the zero map otherwise (If  $\mathcal{A}'^{r, s} = 0$ or $\mathcal{A}'^{r+1, s} = 0$). 

The vertical differential $\overline{\partial}'$ of $\mathcal{A}'$ is defined in the same way.
\end{Definition}
Some cuttings might not fundamentaly change the double complex. For instance, if the double complex is bounded, a cutting from above at degree $p$ will not do anything if $p$ is large enough. 
\begin{Definition}
Let $\mathcal{A}$ be a double complex and consider one of its cutting $\mathcal{A}'$ (from above, the bottom, the right or the left). We say that the cutting is {\sl strict} if $\mathcal{A}'$ is not isomorphic to $\mathcal{A}$ as a double complex.
\end{Definition}

\begin{Example}
We show the zigzag $Z^5_{4, 4}$, before and after being cut from above at degree $2$ and from the right at degree $3$.
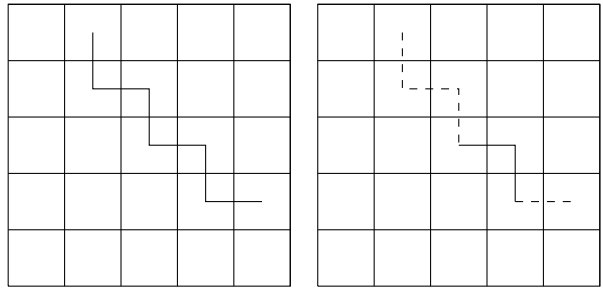
\[
\begin{tikzpicture}[inner sep=0pt, outer sep=0pt, remember picture]
\draw (0, 0) -- (3.75, 0) -- (3.75, 3.75) -- (0, 3.75) -- (0, 0);
\draw (0, .75) -- (3.75, .75);
\draw (0, 1.5) -- (3.75, 1.5);
\draw (0, 2.25) -- (3.75, 2.25);
\draw (0, 3) -- (3.75, 3);
\draw (0, 3.75) -- (3.75, 3.75);
\draw (.75, 0) -- (.75, 3.75);
\draw (1.5, 0) -- (1.5, 3.75);
\draw (2.25, 0) -- (2.25, 3.75);
\draw (3, 0) -- (3, 3.75);
\draw (3.75, 0) -- (3.75, 3.75);
\draw(1.125, 3.375) -- (1.125, 2.625) -- (1.875, 2.625) -- (1.875, 1.875) -- (2.625, 1.875) -- (2.625, 1.125) -- (3.375, 1.125);
\end{tikzpicture}
\quad
\begin{tikzpicture}[inner sep=0pt, outer sep=0pt, remember picture]
\draw (0, 0) -- (3.75, 0) -- (3.75, 3.75) -- (0, 3.75) -- (0, 0);
\draw (0, .75) -- (3.75, .75);
\draw (0, 1.5) -- (3.75, 1.5);
\draw (0, 2.25) -- (3.75, 2.25);
\draw (0, 3) -- (3.75, 3);
\draw (0, 3.75) -- (3.75, 3.75);
\draw (.75, 0) -- (.75, 3.75);
\draw (1.5, 0) -- (1.5, 3.75);
\draw (2.25, 0) -- (2.25, 3.75);
\draw (3, 0) -- (3, 3.75);
\draw (3.75, 0) -- (3.75, 3.75);
\draw[dashed] (1.125, 3.375) -- (1.125, 2.625) -- (1.875, 2.625) -- (1.875, 1.875) ;
\draw (1.875, 1.875) -- (2.625, 1.875) -- (2.625, 1.125);
\draw[dashed] (2.625, 1.125) -- (3.375, 1.125);
\end{tikzpicture}
\]
Notice how the total degree of the first zigzag is $5$, whereas the total degree of the zigzag obtained after doing the cutting is $4$.
\end{Example}

The dimension $h^{k}_{p, q}$ of $H^k_{p, q}$ can be computed in the following way: We start by choosing a minimal representative of the double complex (equivalently, we remove all the squares). Then, we distinguish two cases:
\begin{itemize}
\item If $k \leqslant p + q$, then cut the complex from above at degree $q$ and from the right at degree $p$. The dimension $h^{k}_{p, q}$ is then simply equal to the $k^{\text{th}}$ Betti number of the new double complex $\tilde{A}$:
\[
h^{k}_{p, q} = b^k(\tilde{A}).
\]
\item If $k > p+q$, then cut the complex from below at degree $q$ and from the left at degree $p$. The dimension $h^{k}_{p, q}$ is then simply equal to the $k+1^{\text{th}}$ Betti number of the new double complex $\tilde{A}$:
\[
h^{k}_{p, q} = b^{k+1}(\tilde{A}).
\]
\end{itemize}

We are interested about the information carried by the Bigolin numbers concerning the zigzags appearing in the decomposition of the double complex. There are two types of zigzags, the odd zigzags and the even ones. We recall that an odd zigzag has its components in two total degrees only. It is said to have total degree $n$ if most of its components are in total degree $n$. We refer to Definition~\ref{odd_zigzags} and to Notations~\ref{notations_zigzags} for odd zigzags, which are the same notations as in~\cite{Ste}.

We will show that the Bigolin numbers are enough to determine the odd zigzags of a double complex. This will be a consequence of the following lemma:
\begin{Lemma}
\label{lemme1}
For every integer $p, q$ and $k$, the odd zigzags contributing to the dimension $h^k_{p, q}$ of $H^k_{p, q}$ are either of the form $Z^{u, v}_k$ with $u \leqslant p$ and $v \leqslant q$, or of the form $Z^{u, v}_{k+1}$ with $u \geqslant p+1$ and $v \geqslant q+1$.
\end{Lemma}

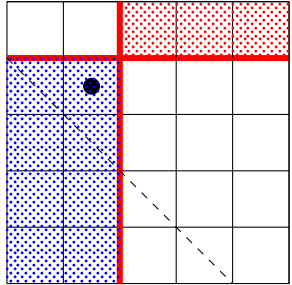
\begin{Example}
\[
\begin{tikzpicture}[inner sep=0pt, outer sep=0pt, remember picture]
\draw (0, 0) -- (3.75, 0) -- (3.75, 3.75) -- (0, 3.75) -- (0, 0);
\draw (0, .75) -- (3.75, .75);
\draw (0, 1.5) -- (3.75, 1.5);
\draw (0, 2.25) -- (3.75, 2.25);
\draw [line width=.8mm, red] (0, 3) -- (3.75, 3);
\draw (0, 3.75) -- (3.75, 3.75);
\draw (.75, 0) -- (.75, 3.75);
\draw [line width=.8mm, red] (1.5, 0) -- (1.5, 3.75);
\draw (2.25, 0) -- (2.25, 3.75);
\draw (3, 0) -- (3, 3.75);
\draw (3.75, 0) -- (3.75, 3.75);
\filldraw (1.125, 2.625) circle (3pt);
\draw [dashed] (0, 3) -- (3, 0);
\fill [pattern=crosshatch dots, pattern color=red]
    (1.5,3.75) rectangle (3.75, 3);
\fill [pattern=crosshatch dots, pattern color=blue]
    (1.5,3) rectangle (0, 0);
\end{tikzpicture}
\]
For the bidegree $(1, 3)$ (black dot), and the total degree $3$ (dashed line), the odd zigzags of total degree $3$ contributing to $h^3_{1, 3}$ are of the form $Z^{u, v}_3$ with $u \leqslant 1$ and $v \leqslant 3$. These bidegrees correspond to the dashed blue area in the picture. Furthermore, the odd zigzag of total degree $4$ contributing to $h^{3}_{1, 3}$ are of the form $Z^{u, v}_{4}$ with $u > 1$ and $v > 3$. These bidegrees correspond to the dashed red area in the picture.
\end{Example}

In order to prove the lemma, we state some facts that can be directly verified on some pictures:
\begin{Property}
Let $Z$ be an odd zigzag of total degree $l$:
\begin{itemize}
\item If $Z$ is looking down, then any cutting from above of from the right turns it into an odd zigzag of total degree $l$.
\item If $Z$ is looking up, then any strict cutting from the top or from the right turns it into an even zigzag.
\item If $Z$ is looking up, then any strict cutting from the top followed by a strict cutting from the right turns it into an odd zigzag of total degree $l-1$.
\end{itemize}
\end{Property}
We are now able to prove the Lemma:
\begin{proof}[proof of Lemma~\ref{lemme1}]
We first consider the case where $k \leqslant p+q$. Let $u, v$ and $l$ be integers such that $h^k_{p, q}(Z_l^{u, v}) \neq 0$. The zigzag $Z_l^{u, v}$ is, after a cutting from above and from the right, of total degree $k$. This means that $Z_{l}^{u, v}$ is of total degree $k-1, k$ or $k+1$. However, if its total degree were $k-1$, then the zigzag would be looking down. Since cutting from above and from the right does not change the total degree of such a zigzag, we are not in that case. We deduce that the total degree of $Z_l^{u, v}$ is $l = k$ or $l = k+1$.

Suppose that $l = k$. If $u+v = k$, then $Z^{u, v}_k$ is a dot and clearly $h^k_{p, q}(Z^{u, v}_k) \neq 0$ if and only if $u \leqslant p$ and $v \leqslant q$. If $u+v < k$, then $Z^{u, v}_k$ is looking down, and by the Property any cutting from the top and the right either turns it into an odd zigzag of total degree $k$, or it vanishes completely. For the zigzag not to vanish completely, it needs to have one component of bidegree $(r, s)$ with $r\leqslant p$ and $s \leqslant q$, or equivalently we must have $u \leqslant p$ and $v \leqslant q$. Now suppose that $u+v > k$. $Z^{u, v}_k$ is looking up. By the Property, $h^k_{p, q}(Z^{u, v}_k) \neq 0$ if and only if the cut from the top at degree $q$ and the cut from the right at degree $p$ are both not strict. Equivalently, we must have $u \leqslant p$ and $v \leqslant q$.

Suppose now that $l = k+1$. Since $h^k_{p, q}(Z^{u, v}_{k+1}) \neq 0$, $Z^{u, v}_{k+1}$ has components in total degree $k$ and is thus looking up. By the Property, we need the cutting from the top at degree $q$ and from the right at degree $p$ to be both strict. Equivalently, we must have $u \geqslant p$ and $v \geqslant q$.

The case $k > p+q$ can be reduced o the previous case, by using the relation
\[
h^k_{p, q}(Z^{u, v}_l)
=
h^{-k-1}_{-p-1, -q-1}(Z^{-u,-v}_{-l}).
\]
\end{proof}

We will now state a similar lemma for even zigzags. We recall that there are two types of even zigzags, the horizontal ones and the vertical ones:
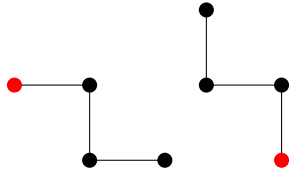
\[
\begin{tikzpicture}
\draw (1, 0) -- (0, 0) -- (0, 1) -- (-1, 1);
\fill [black] (1,0) circle [radius=.1];
\fill [black] (0,0) circle [radius=.1];
\fill [black] (0,1) circle [radius=.1];
\fill [red] (-1,1) circle [radius=.1];
\end{tikzpicture}
\quad
\begin{tikzpicture}
\draw (1, -1) -- (1, 0) -- (0, 0) -- (0, 1);
\fill [black] (1,0) circle [radius=.1];
\fill [red] (1,-1) circle [radius=.1];
\fill [black] (0,0) circle [radius=.1];
\fill [black] (0,1) circle [radius=.1];
\end{tikzpicture}
\]
In both pictures, the red dot will be refered to as the {\sl starting point}. An even zigzag will be denoted by $Z^{p, q}_{\epsilon, l}$, where $(p, q)$ is the bidegree of the starting point, $2l$ is the length of the zigzag, and $\epsilon = 1$ if the zigzag is horizontal and $\epsilon = 2$ if the zigzag is vertical. For example, in the above picture, the zigzag on the left is denoted $Z^{0, 1}_{1, 2}$ and the zigzag on the right $Z^{1, 0}_{2, 2}$. Here are some properties about the behaviour of even, horizontal zigzags when cut:
\begin{Property}
\label{cuts_even}
Let $Z$ be an even, horizontal zigzag.
\begin{itemize}
\item If we cut $Z$ from above or from below, we obtain an even zigzag (which might have length $0$).
\item If we cut $Z$ from the right, we are in one of the following cases:
\begin{itemize}
\item None of the components of $Z$ were cut off, and $Z$ is left unchanged.
\item All of the components of $Z$ were cut off, and we obtain an empty zigzag.
\item Some, but not all, of the components of $Z$ were cut off, and we obtain an odd zigzag whose total degree is equal to the total degree of the starting point of $Z$.
\end{itemize}
\item If we cut $Z$ from the left, we are in one of the following cases:
\begin{itemize}
\item None of the components of $Z$ were cut off, and $Z$ is left unchanged.
\item All of the components of $Z$ were cut off, and we obtain an empty zigzag.
\item Some, but not all, of the components of $Z$ were cut off, and we obtain an odd zigzag whose total degree is equal to $1$ plus the total degree of the starting point of $Z$.
\end{itemize}
\end{itemize}
\end{Property}
We will now state a Lemma that describes the horizontal even zigzag contributing to some Bigolin number.
\begin{Lemma}
\label{lemme2}
Let $k \in \mathbb{Z}$ be some degree and $(p, q)$ a bidegree. Then an even, horizontal zigzag $Z^{u, v}_{1, l}$ contributes to the Bigolin number $h^k_{p, q}$ if and only if $u + v = k$ and $u \leqslant p < u+l$. 
\end{Lemma}
\begin{proof}
Suppose first that $k \leqslant p+q$. By the Property~\ref{cuts_even}, a cutting from above and from the right turns $Z^{u, v}_{1, l}$ into an odd zigzag of total degree $k$ if and only if the starting point $(u, v)$ is of total degree $u+v=k$ and some components are cut off from the right, but not all. This last condition is equivalent to $u \leqslant p < u+l$. Note that in that case the zigzag has a nonzero component in bidegree $(p, v-(p-u))$, which does not vanish after the cutting from above as $v-(p-u) = k-p \leqslant q$.

Now suppose that $k > p+q$. Again by the Property~\ref{cuts_even}, a cutting from below and from the left turns $Z^{u, v}_{1, l}$ into an odd zigzag of total degree $k+1$ if and only if the starting point $(u, v)$ is of total degree $u+v=k$ and some components are cut off from the left, but not all. This last condition is equivalent to $u \leqslant p < u+l$. Note that in that case the zigzag has a nonzero component in bidegree $(p+1, v-(p+1-u)+1)$, which does not vanish after the cutting from below as $v-(p+1-u)+1=k-p>q$.
\end{proof}

For every degree $k$ and bidegree $(p, q)$, we denote the number of odd zigzags of the form $Z^{p, q}_k$ in the decomposition of a double complex $\mathcal{A}$ by $b^k_{p, q}$. Furthermore, we define the number $\tilde{h}^k_{p, q}$ as follows:
\begin{equation}
\label{eq_bigolin}
\tilde{h}^k_{p, q}
=
h^k_{p-1, q-1} - h^k_{p-1, q} - h^k_{p, q-1} + h^k_{p, q}.
\end{equation}
\begin{Theorem}
\label{submain}
For every bounded double complex $\mathcal{A}$ the following holds:
\begin{equation}
\label{eq_main}
\tilde{h}^{k}_{p, q}
-
\tilde{h}^{k-1}_{p, q}
+
\tilde{h}^{k-2}_{p, q}
-
\cdots
=
b^{k+1}_{p, q}.
\end{equation}
\end{Theorem}
\begin{proof}
By Lemma~\ref{lemme1} we know which zigzags are represented in each term of the sum defining $\tilde{h}^{k}_{p, q}$. In fact, the sum is non zero only for $Z^{p, q}_{k+1}$ and $Z^{p, q}_k$. For the zigzags of total degree $k+1$, this is easily seen in the following picture, where the dashed area represent the bidegrees $(r, s)$ for which the zigzag $Z^{r, s}_{k+1}$ contributes to the Bigolin numbers of the sum.
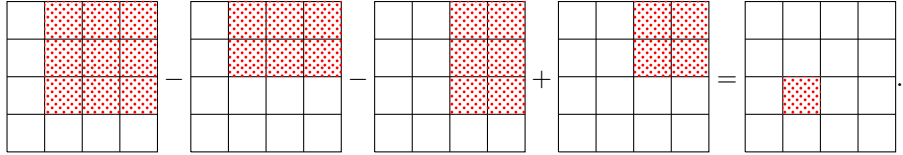
\[
\begin{tikzpicture}[baseline=5.8ex, inner sep=0pt, outer sep=0pt, remember picture]
\draw (0, 0) -- (2, 0) -- (2, 2) -- (0, 2) -- (0, 0);
\draw (0, .5) -- (2, .5);
\draw (0, 1) -- (2, 1);
\draw (0, 1.5) -- (2, 1.5);
\draw (.5, 0) -- (.5, 2);
\draw (1, 0) -- (1, 2);
\draw (1.5, 0) -- (1.5, 2);
\fill [pattern=crosshatch dots, pattern color=red]
    (.5, 2) rectangle (2, .5);
\end{tikzpicture}
-
\begin{tikzpicture}[baseline=5.8ex, inner sep=0pt, outer sep=0pt, remember picture]
\draw (0, 0) -- (2, 0) -- (2, 2) -- (0, 2) -- (0, 0);
\draw (0, .5) -- (2, .5);
\draw (0, 1) -- (2, 1);
\draw (0, 1.5) -- (2, 1.5);
\draw (.5, 0) -- (.5, 2);
\draw (1, 0) -- (1, 2);
\draw (1.5, 0) -- (1.5, 2);
\fill [pattern=crosshatch dots, pattern color=red]
    (.5, 2) rectangle (2, 1);
\end{tikzpicture}
-
\begin{tikzpicture}[baseline=5.8ex, inner sep=0pt, outer sep=0pt, remember picture]
\draw (0, 0) -- (2, 0) -- (2, 2) -- (0, 2) -- (0, 0);
\draw (0, .5) -- (2, .5);
\draw (0, 1) -- (2, 1);
\draw (0, 1.5) -- (2, 1.5);
\draw (.5, 0) -- (.5, 2);
\draw (1, 0) -- (1, 2);
\draw (1.5, 0) -- (1.5, 2);
\fill [pattern=crosshatch dots, pattern color=red]
    (1, 2) rectangle (2, .5);
\end{tikzpicture}
+
\begin{tikzpicture}[baseline=5.8ex, inner sep=0pt, outer sep=0pt, remember picture]
\draw (0, 0) -- (2, 0) -- (2, 2) -- (0, 2) -- (0, 0);
\draw (0, .5) -- (2, .5);
\draw (0, 1) -- (2, 1);
\draw (0, 1.5) -- (2, 1.5);
\draw (.5, 0) -- (.5, 2);
\draw (1, 0) -- (1, 2);
\draw (1.5, 0) -- (1.5, 2);
\fill [pattern=crosshatch dots, pattern color=red]
    (1, 2) rectangle (2, 1);
\end{tikzpicture}
=
\begin{tikzpicture}[baseline=5.8ex, inner sep=0pt, outer sep=0pt, remember picture]
\draw (0, 0) -- (2, 0) -- (2, 2) -- (0, 2) -- (0, 0);
\draw (0, .5) -- (2, .5);
\draw (0, 1) -- (2, 1);
\draw (0, 1.5) -- (2, 1.5);
\draw (.5, 0) -- (.5, 2);
\draw (1, 0) -- (1, 2);
\draw (1.5, 0) -- (1.5, 2);
\fill [pattern=crosshatch dots, pattern color=red]
    (.5, 1) rectangle (1, .5);
\end{tikzpicture}.
\]
Furthermore, by Lemma~\ref{lemme2}, we see that $h^k_{p-1, q-1}(Z) = h^k_{p-1, q}(Z)$ and that $h^k_{p, q-1}(Z) = h^k_{p, q}(Z)$ for all even, horizontal zigzags $Z$. Thus $\tilde{h}^{k}_{p, q}$ vanishes on even, horizontal zigzags. The same is true for vertical even zigzags and can be proven in a similar way.
We deduce that
\[
\tilde{h}^{k}_{p, q} = b^{k+1}_{p, q} + b^k_{p, q}.
\]
Therefore
\[
\tilde{h}^{k}_{p, q}
-
\tilde{h}^{k-1}_{p, q}
+
\tilde{h}^{k-2}_{p, q}
-
\cdots
=
b^{k+1}_{p, q},
\]
where only a finite number of the terms in the sum are nonzero.
\end{proof}

\begin{Corollary}
Two bounded double complexes $\mathcal{A}_1$ and $\mathcal{A}_2$ are isomorphic if and only if they share the same Frölicher spectral sequences and the same Bigolin numbers.
\end{Corollary}

\begin{Remark}
If a double complex $\mathcal{A}^{\bullet, \bullet}$ lies in the first quadrant, meaning that $\mathcal{A}^{p, q} = 0$ if $p < 0$ or $q < 0$, then $h^{l}_{p, q} = 0$ for all $l < 0$. Thus:
\[
\tilde{h}^k_{p, q} - 
\tilde{h}^{k-1}_{p, q} +
\tilde{h}^{k-2}_{p, q} -  
\cdots
+(-1)^{k}
\tilde{h}^{0}_{p, q}
=
b^{k+1}_{b, q}.
\]
This is the case for the double complexes assciated to compact complex manifolds.
\end{Remark}

To conclude this section, we shall present two examples where the Bigolin numbers fail to distinguish two non isomorphic double complexes. Let $\phi_{\mathcal{B}} : \mathcal{U}^{\text{bounded, fin}} \longrightarrow \mathbb{Z}^{\mathbb{N}}$ be the morphism which associates to every double complex $\mathcal{A}$ the array
\[
\phi_{\mathcal{B}}
(\mathcal{A})
=
(b_k^{p, q})_{p, q, k \in \mathbb{Z}}
\] 
of its Bigolin numbers.
\begin{Example}
In dimension $3$, the kernel of $\phi_{\mathcal{B}}$ is of dimension $2$ and has the following generators:
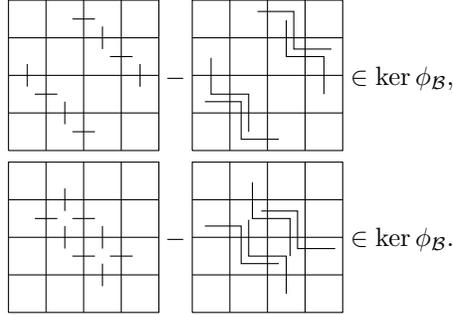
\[
\aligned
\begin{tikzpicture}[baseline=5.8ex, inner sep=0pt, outer sep=0pt, remember picture]
\draw (0, 0) -- (2, 0) -- (2, 2) -- (0, 2) -- (0, 0);
\draw (0, .5) -- (2, .5);
\draw (0, 1) -- (2, 1);
\draw (0, 1.5) -- (2, 1.5);
\draw (.5, 0) -- (.5, 2);
\draw (1, 0) -- (1, 2);
\draw (1.5, 0) -- (1.5, 2);
\draw (0.25, 0.85) -- (0.25, 1.15);
\draw (0.35, 0.75) -- (0.65, 0.75);
\draw (0.75, 0.35) -- (0.75, 0.65);
\draw (0.85, 0.25) -- (1.15, 0.25);
\draw (0.85, 1.75) -- (1.15, 1.75);
\draw (1.25, 1.65) -- (1.25, 1.35);
\draw (1.35, 1.25) -- (1.65, 1.25);
\draw (1.75, 1.15) -- (1.75, 0.85);
\end{tikzpicture}
-
\begin{tikzpicture}[baseline=5.8ex, inner sep=0pt, outer sep=0pt, remember picture]
\draw (0, 0) -- (2, 0) -- (2, 2) -- (0, 2) -- (0, 0);
\draw (0, .5) -- (2, .5);
\draw (0, 1) -- (2, 1);
\draw (0, 1.5) -- (2, 1.5);
\draw (.5, 0) -- (.5, 2);
\draw (1, 0) -- (1, 2);
\draw (1.5, 0) -- (1.5, 2);
\node at (0.15, 0.65) (a) {};
\node at (0.25, 1.25) (b) {};
\node at (0.85, 1.85) (c) {};
\node at (1.25, 1.75) (d) {};
\draw (a) -- +(0.5, 0) -- +(0.5, -0.5) -- +(1, -0.5);
\draw (b) -- +(0, -0.5) -- +(0.5, -0.5) -- +(0.5, -1);
\draw (c) -- +(0.5, 0) -- +(0.5, -0.5) -- +(1, -0.5);
\draw (d) -- +(0, -0.5) -- +(0.5, -0.5) -- +(0.5, -1);
\end{tikzpicture}
\in
\ker \phi_{\mathcal{B}}, \\
\begin{tikzpicture}[baseline=5.8ex, inner sep=0pt, outer sep=0pt, remember picture]
\draw (0, 0) -- (2, 0) -- (2, 2) -- (0, 2) -- (0, 0);
\draw (0, .5) -- (2, .5);
\draw (0, 1) -- (2, 1);
\draw (0, 1.5) -- (2, 1.5);
\draw (.5, 0) -- (.5, 2);
\draw (1, 0) -- (1, 2);
\draw (1.5, 0) -- (1.5, 2);
\draw (0.35, 1.25) -- (0.65, 1.25);
\draw (0.85, 1.25) -- (1.15, 1.25);
\draw (0.75, 1.65) -- (0.75, 1.35);
\draw (0.75, 1.15) -- (0.75, 0.85);
\draw (0.85, 0.75) -- (1.15, 0.75);
\draw (1.35, 0.75) -- (1.65, 0.75);
\draw (1.25, 1.15) -- (1.25, 0.85);
\draw (1.25, 0.65) -- (1.25, 0.35);
\end{tikzpicture}
-
\begin{tikzpicture}[baseline=5.8ex, inner sep=0pt, outer sep=0pt, remember picture]
\draw (0, 0) -- (2, 0) -- (2, 2) -- (0, 2) -- (0, 0);
\draw (0, .5) -- (2, .5);
\draw (0, 1) -- (2, 1);
\draw (0, 1.5) -- (2, 1.5);
\draw (.5, 0) -- (.5, 2);
\draw (1, 0) -- (1, 2);
\draw (1.5, 0) -- (1.5, 2);
\node at (0.15, 1.15) (a) {};
\node at (0.75, 1.25) (b) {};
\node at (0.9, 1.35) (c) {};
\node at (0.8, 1.75) (d) {};
\draw (a) -- +(0.5, 0) -- +(0.5, -0.5) -- +(1, -0.5);
\draw (b) -- +(0, -0.5) -- +(0.5, -0.5) -- +(0.5, -1);
\draw (c) -- +(0.5, 0) -- +(0.5, -0.5) -- +(1, -0.5);
\draw (d) -- +(0, -0.5) -- +(0.5, -0.5) -- +(0.5, -1);
\end{tikzpicture}
\in
\ker \phi_{\mathcal{B}}.
\endaligned
\]
However, these generators have non vanishing Dolbeault cohomologies. Therefore, in dimension $3$, a double complex can be decomposed if the Hodge diamond and the Bigolin numbers are known (see also~\cite{Piovani}). In the next example we will see that this is not true anymore if we increase the dimension.
\end{Example}
\begin{Example}
The following two double complexes
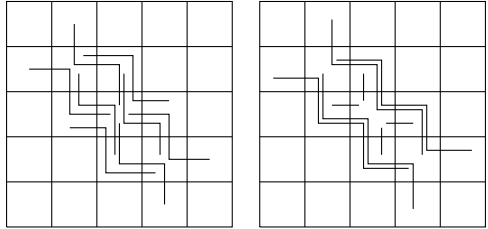
\[
\begin{tikzpicture}[inner sep=0pt, outer sep=0pt, remember picture, scale=.6]
\draw (0, 0) -- (5, 0) -- (5, 5) -- (0, 5) -- (0, 0);
\draw (0, 1) -- (5, 1);
\draw (0, 2) -- (5, 2);
\draw (0, 3) -- (5, 3);
\draw (0, 4) -- (5, 4);
\draw (1, 0) -- (1, 5);
\draw (2, 0) -- (2, 5);
\draw (3, 0) -- (3, 5);
\draw (4, 0) -- (4, 5);
\draw (.5, 3.5) -- (1.4, 3.5) -- (1.4, 2.5) -- (2.3, 2.5);
\draw (1.5, 4.5) -- (1.5, 3.6) -- (2.5, 3.6) -- (2.5, 2.7);
\draw (2.5, 2.3) -- (2.5, 1.4) -- (3.5, 1.4) -- (3.5, .5);
\draw (2.7, 2.5) -- (3.6, 2.5) -- (3.6, 1.5) -- (4.5, 1.5);
\draw (1.6, 3.4) -- (1.6, 2.7) -- (2.4, 2.7) -- (2.4, 1.6);
\draw (2.6, 3.4) -- (2.6, 2.3) -- (3.4, 2.3) -- (3.4, 1.6);
\draw (1.4, 2.2) -- (2.2, 2.2) -- (2.2, 1.2) -- (3.3, 1.2);
\draw (1.7, 3.8) -- (2.8, 3.8) -- (2.8, 2.8) -- (3.6, 2.8);
\end{tikzpicture}
\quad 
\begin{tikzpicture}[inner sep=0pt, outer sep=0pt, remember picture, scale=.6]
\draw (0, 0) -- (5, 0) -- (5, 5) -- (0, 5) -- (0, 0);
\draw (0, 1) -- (5, 1);
\draw (0, 2) -- (5, 2);
\draw (0, 3) -- (5, 3);
\draw (0, 4) -- (5, 4);
\draw (1, 0) -- (1, 5);
\draw (2, 0) -- (2, 5);
\draw (3, 0) -- (3, 5);
\draw (4, 0) -- (4, 5);
\draw (.3, 3.3) -- (1.3, 3.3) -- (1.3, 2.3) -- (2.3, 2.3) -- (2.3, 1.3) -- (3.3, 1.3);
\draw (1.4, 3.4) -- (1.4, 2.4) -- (2.4, 2.4) -- (2.4, 1.4) -- (3.4, 1.4) -- (3.4, .4);
\draw (4.7, 1.7) -- (3.7, 1.7) -- (3.7, 2.7) -- (2.7, 2.7) -- (2.7, 3.7) -- (1.7, 3.7);
\draw (3.6, 1.6) -- (3.6, 2.6) -- (2.6, 2.6) -- (2.6, 3.6) -- (1.6, 3.6) -- (1.6, 4.6);
\draw (1.6, 2.7) -- (2.2, 2.7);
\draw (3.4, 2.3) -- (2.8, 2.3);
\draw (2.3, 3.4) -- (2.3, 2.8);
\draw (2.7, 1.6) -- (2.7, 2.2);
\end{tikzpicture}
\]
share the same Bigolin numbers. Both have vanishing Betti numbers, as there are only even zigzags. Furthermore, the two double complexes are locally similar, which implies that they have the same Varouchas cohomologies. In particular, they share the same Hodge, Bott-Chern and Aeppli numbers.
\end{Example}

%%%%%%%%%%%%%%%%%%%%%%%%%%%%%%%%%%%%%%
\section{Lie algebras}
\label{lie}
In this section we will discuss about {\sl nilmanifolds}, which are a special class of complex manifolds. 
Recall that a {\sl nilpotent Lie algebra} is a Lie algebra $\mathfrak{g}$ whose lower central series
\[
\aligned
\mathfrak{g}^1 &= \mathfrak{g}, \\
\mathfrak{g}^{k+1} &= [\mathfrak{g}, \mathfrak{g}^{k}],
\endaligned
\]
satisfies $\mathfrak{g}^N = 0$ for some $N$.

By a Theorem of Mal'tsev \cite{AM}, the simply connected nilpotent Lie group $G$ associated to $\mathfrak{g}$ admits a lattice $\Gamma$ if and only if the Lie algebra $\mathfrak{g}$ has a $\mathbb{Q}$-structure, that means that there exists a Lie algebra $\mathfrak{g}'$ over the rationals such that 
\[
\mathfrak{g}
=
\mathfrak{g}' \otimes \mathbb{R}.
\]
If this holds, then we can consider the compact manifold $M = G / \Gamma$.

%%%%%%%%%%%%%%%%%%%%%%%%%%%%%%%%%%%%%
\subsection{The Chevalley-Eilenberg complex}
%%%%%%%%%%%%%%%%%%%%%%%%%%%%%%%%%%%%%
In this section $\mathfrak{g}$ will always denote a nilpotent Lie algebra, and $\mathfrak{g}^*$ its dual. The space of $k$-forms on $\mathfrak{g}$ is denoted by $\Lambda^k \mathfrak{g}^*$. 
We define a map 
$d: \mathfrak{g}^* \longrightarrow \Lambda^{2} \mathfrak{g}^*$
which sends a $1$-form $\omega$ to the $2$-form $d \omega$ defined by:
\[
d\omega(X, Y)
=
-\omega ([X, Y]).
\]
We extend $d$ so that it becomes a map of the graded algebra $\Lambda^{\bullet} \mathfrak{g}^*$ of degree $+1$ which satisfies the graded Leibniz rule:
\[
d(\alpha \wedge \beta)
=
(d\alpha) \wedge \beta
+
(-1)^{|\alpha|}\alpha \wedge (d\beta).
\]
Thus, the graded algebra $\Lambda^{\bullet}\mathfrak{g}^*$ becomes a differential algebra.
The $k$-forms in $\Lambda^k \mathfrak{g}^*$ correspond to left invariant forms of the associated Lie group $G$. Suppose that $G$ has a lattice $\Gamma$, so that $M = G/ \Gamma$ is a compact nilmanifold. The Nomizu Theorem~\cite{Nom} states that the inclusion of the left invariant forms into the space of all forms induces a quasi-isomorphism in the De Rham cohomology.
\begin{Theorem}
Let $M$ be a compact nilmanifold with Lie algebra $\mathfrak{g}$. We have an isomorphism
\[
H^{\bullet}(\mathfrak{g}, \mathbb{R})
\simeq
H^{\bullet}_{\text{dR}}(M, \mathbb{R}).
\]
\end{Theorem}

An {\sl almost complex strucure} on $\mathfrak{g}$ is a linear map $J : \mathfrak{g} \longrightarrow \mathfrak{g}$ such that
$
J^2
=
-Id$.
We also denote by $J$ the dual map:
\[
\begin{array}{cccc}
J : & \mathfrak{g}^{\ast} &\longrightarrow & \mathfrak{g}^{\ast}\\
& \omega & \mapsto &\omega \circ J
\end{array},
\]
which has eigenvalues $i$ and $-i$. Thus, the space of complex forms
$
\mathfrak{g}_{\mathbb{C}}^{\ast}
=
\mathfrak{g}^{\ast}
\otimes
\mathbb{C}
$
writes as the sum of the eigenspaces :
\[
\mathfrak{g}_{\mathbb{C}}^{\ast}
=
{\mathfrak{g}^*}^{1, 0}
\oplus
{\mathfrak{g}^*}^{0, 1}.
\]
The almost complex structure on $\mathfrak{g}$ induces an almost complex structure on the Lie group $G$. 

Furthermore, by the Newlander-Nirenberg Theorem (\cite{NN}), the almost complex structure is induced by a complex structure if the following integrability condition is satisfied:
\[
J[X, Y]
=
J[JX, Y]
+
J[X, JY]
+
[JX, JY]
\quad \quad
\forall X, Y \in \mathfrak{g}.
\]
An integrable almost complex structure $J$ will be simply called a {\sl complex structure}.
For a complex structure $J$, the following holds:
\[
d{\mathfrak{g}^*}^{1, 0}
\subset
{\mathfrak{g}^*}^{1, 1}
\oplus
{\mathfrak{g}^*}^{2, 0}.
\]
Hence, $(\Lambda \mathfrak{g}^*, \partial, \overline{\partial})$ is a double complex.

We now give the definition of a {\sl nilpotent complex structure}:

\begin{Definition}
A complex structure $J$ is said to be {\sl nilpotent} if there is a basis $\{\omega^1, \omega^2, \dots, \omega^n\}$ of ${\mathfrak{g}^*}^{1, 0}$ such that for all $i \in \llbracket 1, n \rrbracket$:
\[
d\omega^{i+1}
\in
\Lambda^2 <\omega^1, \omega^2, \dots, \omega^i, \overline{\omega^1}, \overline{\omega^2}, \dots, \overline{\omega^i}>.
\]
\end{Definition}
For such a basis, we have in particular $d\omega^1 = 0$ and $d\omega^2 = A\omega^{1\overline{1}}$, for some constant $A \in \mathbb{C}$.

Special cases of nilpotent complex structures are given by {\sl abelian complex structures} ($d{\mathfrak{g}^*}^{1, 0} \subset \Lambda^{1, 1}\mathfrak{g}^*$) and {\sl complex parallelizable structures} ($d{\mathfrak{g}^*}^{1, 0} \subset \Lambda^{2, 0}\mathfrak{g}^*$).

Forms in $\Lambda^{p, q}\mathfrak{g}^*$ correspond to left invariant forms of the Lie group $G$. By a Theorem of Console and Fino (\cite{CF}), the inclusion of left invariant forms into the space forms of $M = G / \Gamma$ induces an isomorphism in Dolbeault cohomology.
\[
H_{\overline{\partial}}^{p, q}(\mathfrak{g})
\simeq
H_{\overline{\partial}}^{p, q}(M).
\]
Therefore, the inclusion of left invariant forms induces an $E_1$-isomorphism (an isomorphism on the first page of the Frölicher spectral sequence), which implies that the zigzags in the double complex of left invariant forms and the ones in the double complex of forms of $M$ have same multiplicities (\cite{Ste2}). In particular, Bott-Chern, Aeppli, Varouchas and Bigolin cohomologies can all be computed by the means of left invariant forms.
%%%%%%%%%%%%%%%%%%%%%%%%%%%%%%%%%%%%%%%%
\subsection{Lie algebras of dimension $6$}
\label{lie6}
%%%%%%%%%%%%%%%%%%%%%%%%%%%%%%%%%%%%%%%%
Nilpotent complex structures on $6$-dimensonal, nilpotent Lie algebras were classified in \cite{COUV}. We are interested in the decomposition of their double complexes. The double complexes satisfy the conditions of Definition~\ref{dc-acc}.

If $(\mathfrak{g}, J)$ is a $6$-dimensional nilpotent Lie algebra, with $J$ a nilpotent complex structure, then there exists a basis $\omega^1, \omega^2, \omega^3$ of ${\mathfrak{g}^*}^{1, 0}$ such that (\cite{Uga}):
\begin{equation}
\label{eq_struct_6}
\begin{array}{ccc}
d\omega^1 = 0, \\
d\omega^2 = \epsilon \omega^{1\overline{1}}, \\
d\omega^3 = \rho \omega^{12} + (1-\epsilon)A\omega^{1\overline{1}} + B\omega^{1\overline{2}} + C\omega^{2\overline{1}} + (1-\epsilon)D\omega^{2\overline{2}},
\end{array}
\end{equation}
with $A, B, C, D \in \mathbb{C}$ and $\epsilon, \rho \in \{0, 1\}$.

We will start by making some comments on the squares in a $6$-dimensional, nilpotent Lie algebra with nilpotent complex structure.
\begin{Lemma}
Let $\mathfrak{g}$ be a nilpotent Lie algebra of dimension $6$ with nilpotent complex structure $J$. If a square $S$ is a direct summand in the Chevalley-Eilenberg complex $\Lambda^{\bullet, \bullet} \mathfrak{g}^*$, then it has its components in bidegrees $(1, 1), (1, 2), (2, 1), (2, 2)$.
\end{Lemma}
\begin{proof}
By the symmetries of the double complex, there is only one other spot for the square: the bidegrees $(1, 0), (1, 1), (2, 0), (2, 1)$. However, in bidegree $(1, 0)$, we clearly have $\partial \overline{\partial} \omega^i = 0$ for $i = 1, 2, 3$. 
\end{proof}
\begin{Lemma}
The complex $\Lambda^{\bullet, \bullet}\mathfrak{g}^*$ has at most one square in its decomposition. Furthermore, there is no square in the decomposition if and only if $\epsilon = 1$ and $d\omega^3=0$, or $\epsilon = 0$ and
\begin{equation}
\label{no-square}
\rho^2+|B|^2+|C|^2-2\Re (A\overline{D})
=
0.
\end{equation}
\end{Lemma}
\begin{proof}
Conisder a basis $\{ \omega^1, \omega^2, \omega^3 \}$ such as in~\ref{eq_struct_6}. If $\epsilon = 1$, then a direct computation yields
\[
\overline{\partial}\partial \omega^{3\overline{3}}
=
(\rho^2 + |B|^2 + |C|^2) \omega^{12\overline{12}},
\]
which vanishes if and only if $d\omega^3 = 0$. If $\epsilon = 0$, then we have
\[
\overline{\partial}\partial \omega^{3\overline{3}}
=
(\rho^2 + |B|^2 + |C|^2 - \Re(A\overline{D})) \omega^{12\overline{12}}.
\]
\end{proof}
From this we deduce the
\begin{Corollary}
If the complex $\Lambda^{\bullet, \bullet}\mathfrak{g}^*$ has a square in its decomposition, then it has two dots in bidegree $(0, 1)$.
\end{Corollary}

Consider the space $\overline{\mathcal{U}^{\text{form}}_3}$ of formal double complexes and the map
\[
\begin{array}{cccc}
\chi: &\overline{\mathcal{U}^{\text{form}}_3}: &\longrightarrow &\mathbb{Z}^{\mathbb{N}} \\
& K^{\bullet, \bullet} &\mapsto &(b^{\bullet}, h_{\overline{\partial}}^{\bullet, \bullet}, h_{\text{BC}}^{\bullet, \bullet})
\end{array},
\]
which send a double complex on the array of its Betti, Hodge and Bott-Chern numbers. A computation using SAGE shows that the kernel of $\chi$ is generated by a single element:
\[
\aligned
\mathcal{T}
&=
\mathcal{T}_+ - \mathcal{T}_- \ ,
\\
\text{where} \quad
\mathcal{T}_+
=
\begin{tikzpicture}[baseline=5.7ex, inner sep=0pt, outer sep=0pt, remember picture]
\draw (0, 0) -- (2, 0) -- (2, 2) -- (0, 2) -- (0, 0);
\draw (0, .5) -- (2, .5);
\draw (0, 1) -- (2, 1);
\draw (0, 1.5) -- (2, 1.5);
\draw (.5, 0) -- (.5, 2);
\draw (1, 0) -- (1, 2);
\draw (1.5, 0) -- (1.5, 2);
\draw (.25, 1.25) -- (.25, .75) -- (.75, .75) -- (.75, .25) -- (1.25, .25);
\draw (.75, 1.75) -- (1.25, 1.75) -- (1.25, 1.25) -- (1.75, 1.25) -- (1.75, .75);
\filldraw (.85,.85) circle (1pt);
\filldraw (1.15,.85) circle (1pt);
\filldraw (.85,1.15) circle (1pt);
\filldraw (1.15,1.15) circle (1pt);
\end{tikzpicture}
& \ \ , \quad \quad
\mathcal{T}_-
=
\begin{tikzpicture}[baseline=5.7ex, inner sep=0pt, outer sep=0pt, remember picture]
\draw (0, 0) -- (2, 0) -- (2, 2) -- (0, 2) -- (0, 0);
\draw (0, .5) -- (2, .5);
\draw (0, 1) -- (2, 1);
\draw (0, 1.5) -- (2, 1.5);
\draw (.5, 0) -- (.5, 2);
\draw (1, 0) -- (1, 2);
\draw (1.5, 0) -- (1.5, 2);
\draw (.15, 1.35) -- (.15, .85) -- (.65, .85);
\draw (.85, .65) -- (.85, .15) -- (1.35, .15);
\draw (.65, 1.85) -- (1.15, 1.85) -- (1.15, 1.35);
\draw (1.35, 1.15) -- (1.85, 1.15) -- (1.85, .65);
\draw (.8, 1.1) -- (.8, .8) -- (1.1, .8);
\draw (.9, 1.2) -- (1.2, 1.2) -- (1.2, .9);
\end{tikzpicture}.
\endaligned
\]
Therefore, if $\mathfrak{g}$ and $\mathfrak{g}'$ are two $6$ dimensional nilpotent Lie algebras with nilpotent complex structure $J$ and $J'$, such that $\Lambda^{\bullet, \bullet}\mathfrak{g}$ and $\Lambda^{\bullet, \bullet}\mathfrak{g}'$ are non isomorphic but share the same Betti, Hodge and Bott-Chern numbers, then $\mathcal{T}_+$ is a subcomplex of either $\Lambda^{\bullet, \bullet}\mathfrak{g}$ or $\Lambda^{\bullet, \bullet}\mathfrak{g}'$, and $\mathcal{T}_-$ is a subcomplex of the other double complex.

\begin{Example}
\label{structures_egales}
For the following structure equations :
\[
\aligned
d\omega^1&=0,\\
d\omega^2&=0,\\
d\omega^3&=\omega^{12} + \omega^{1\overline{1}} + D\omega^{2\overline{2}}, 
\endaligned
\]
we obtain the following double complexes, depending on the value of $D$:
\[
\aligned
&D \in \mathbb{R} \ \backslash \ \{-1, 0, 1/2\} \quad \quad & &D \in \Big\{\frac{1}{2} + iy\Big\}_{0<y < \sqrt{\frac{3}{4}}} \\
&
\begin{tikzpicture}[baseline=5.8ex, inner sep=0pt, outer sep=0pt, remember picture]
\draw (0, 0) -- (2, 0) -- (2, 2) -- (0, 2) -- (0, 0);
\draw (0, .5) -- (2, .5);
\draw (0, 1) -- (2, 1);
\draw (0, 1.5) -- (2, 1.5);
\draw (.5, 0) -- (.5, 2);
\draw (1, 0) -- (1, 2);
\draw (1.5, 0) -- (1.5, 2);
\draw (0.25, 1.15) -- (0.25, .6) -- (.6, .6) -- (.6, .25) -- (1.15, .25);
\draw (1.75, .85) -- (1.75, 1.4) -- (1.4, 1.4) -- (1.4, 1.75) -- (.85, 1.75);
\filldraw (.1, .6) circle (1pt);
\filldraw (.1, .7) circle (1pt);
\filldraw (1.9, 1.4) circle (1pt);
\filldraw (1.9, 1.3) circle (1pt);
\filldraw (.6, .1) circle (1pt);
\filldraw (.7, .1) circle (1pt);
\filldraw (1.4, 1.9) circle (1pt);
\filldraw (1.3, 1.9) circle (1pt);
\filldraw (.65, .65) circle (1pt);
\filldraw (.7, .75) circle (1pt);
\filldraw (.75, .65) circle (1pt);
\filldraw (1.35, 1.35) circle (1pt);
\filldraw (1.3, 1.25) circle (1pt);
\filldraw (1.25, 1.35) circle (1pt);
\draw (.35, 1.1) -- (.65, 1.1) -- (.65, .9);
\draw (.35, 1.2) -- (.75, 1.2) -- (.75, .9);
\draw (1.1, .35) -- (1.1, .65) -- (.9, .65);
\draw (1.2, .35) -- (1.2, .75) -- (.9, .75);
\draw (.9, 1.65) -- (.9, 1.35) -- (1.1, 1.35);
\draw (.8, 1.65) -- (.8, 1.25) -- (1.1, 1.25);
\draw (1.65, .9) -- (1.35, .9) -- (1.35, 1.1);
\draw (1.65, .8) -- (1.25, .8) -- (1.25, 1.1);
\filldraw (.6, 1.4) circle (1pt);
\filldraw (.6, 1.3) circle (1pt);
\filldraw (.7, 1.4) circle (1pt);
\filldraw (.7, 1.3) circle (1pt);
\filldraw (1.4, .6) circle (1pt);
\filldraw (1.4, .7) circle (1pt);
\filldraw (1.3, .6) circle (1pt);
\filldraw (1.3, .7) circle (1pt);
\draw (1, 1) circle (4pt);
\filldraw (.2, .2) circle (1pt);
\filldraw (1.8, .2) circle (1pt);
\filldraw (.2, 1.8) circle (1pt);
\filldraw (1.8, 1.8) circle (1pt);
\end{tikzpicture}
& &
\begin{tikzpicture}[baseline=5.8ex, inner sep=0pt, outer sep=0pt, remember picture]
\draw (0, 0) -- (2, 0) -- (2, 2) -- (0, 2) -- (0, 0);
\draw (0, .5) -- (2, .5);
\draw (0, 1) -- (2, 1);
\draw (0, 1.5) -- (2, 1.5);
\draw (.5, 0) -- (.5, 2);
\draw (1, 0) -- (1, 2);
\draw (1.5, 0) -- (1.5, 2);
\draw (.35, 1.1) -- (.65, 1.1) -- (.65, .9);
\draw (.35, 1.2) -- (.75, 1.2) -- (.75, .9);
\draw (1.1, .35) -- (1.1, .65) -- (.9, .65);
\draw (1.2, .35) -- (1.2, .75) -- (.9, .75);
\draw (.9, 1.65) -- (.9, 1.35) -- (1.1, 1.35);
\draw (.8, 1.65) -- (.8, 1.25) -- (1.1, 1.25);
\draw (1.65, .9) -- (1.35, .9) -- (1.35, 1.1);
\draw (1.65, .8) -- (1.25, .8) -- (1.25, 1.1);
\filldraw (.2, .2) circle (1pt);
\filldraw (1.8, .2) circle (1pt);
\filldraw (.2, 1.8) circle (1pt);
\filldraw (1.8, 1.8) circle (1pt);
\filldraw (.1, .6) circle (1pt);
\filldraw (.1, .7) circle (1pt);
\filldraw (1.9, 1.4) circle (1pt);
\filldraw (1.9, 1.3) circle (1pt);
\filldraw (.6, .1) circle (1pt);
\filldraw (.7, .1) circle (1pt);
\filldraw (1.4, 1.9) circle (1pt);
\filldraw (1.3, 1.9) circle (1pt);
\draw (0.25, 1.15) -- (0.25, .75) -- (.6, .75);
\draw (1.75, .85) -- (1.75, 1.25) -- (1.4, 1.25);
\draw (.75, .6) -- (.75, .25) -- (1.15, .25);
\draw (1.25, 1.4) -- (1.25, 1.75) -- (.85, 1.75);
\filldraw (.6, 1.4) circle (1pt);
\filldraw (.6, 1.3) circle (1pt);
\filldraw (.7, 1.35) circle (1pt);
\filldraw (1.4, .6) circle (1pt);
\filldraw (1.4, .7) circle (1pt);
\filldraw (1.3, .65) circle (1pt);
\draw (.85, 1.1) -- (.85, .85) -- (1.1, .85);
\draw (.9, 1.15) -- (1.15, 1.15) -- (1.15, .9); 
\filldraw (.65, .65) circle (1pt);
\filldraw (.75, .75) circle (1pt);
\filldraw (1.35, 1.35) circle (1pt);
\filldraw (1.25, 1.25) circle (1pt);
\end{tikzpicture}
\endaligned
\]
\end{Example}
They both have the same Betti, Hodge and Bott-Chern numbers. 
The underlying Lie algebra of the left complex structure is $\mathfrak{h}_2$ if $D < -\frac{1}{4}$, 
$\mathfrak{h}_4$ if $D = -\frac{1}{4}$
\footnote{The complex structures for $\mathfrak{h}_2$ and $\mathfrak{h}_4$ displayed here are missing in the classification of~\cite{COUV}. A complete classification is to appear in an Erratum.}
and $\mathfrak{h}_5$ if $D > -\frac{1}{4}$.
The underlying Lie algebra of the right complex structure is $\mathfrak{h}_5$. Isomorphic double complexes are obtained
for the Lie algebra $\mathfrak{h}_2$ with a non-abelian nilpotent complex structure associated to $D$ with $\Re (D) = 1$,
and for the Lie algebra $\mathfrak{h}_4$ with the non-abelian nilpotent complex structure associated to $D=1$ (see Table~\cite{COUV}).

In fact, in dimension $6$, these are the only cases of nilpotent complex structures having non isomorphic double complexes, yet having the same Betti, Hodge and Bott-Chern numbers.
\begin{Theorem}
If two $6$-dimensional nilpotent Lie algebras $(\mathfrak{g}, J)$ and $(\mathfrak{g}', J')$ endowed with nilpotent complex structures have the same Betti, Hodge and Bott-Chern numbers but non-isomorphic double complexes, then the Lie algebras are isomorphic to $\mathfrak{h}_2$, $\mathfrak{h}_4$ or $\mathfrak{h}_5$, with complex structures given in Example~\ref{structures_egales}.
\end{Theorem}
\begin{proof}
Denote the associated double complexes by $\Lambda^{\bullet, \bullet}\mathfrak{g}^*$ and $\Lambda^{\bullet, \bullet}\mathfrak{g'}^*$.
If $S_{1, 1}$ denotes the double complex consisting in only one central square, then we can write
\[
\aligned
\Lambda^{\bullet, \bullet}\mathfrak{g}^* = K \oplus \delta S_{1, 1}, \\
\Lambda^{\bullet, \bullet}\mathfrak{g'}^* = K' \oplus \delta ' S_{1, 1}.
\endaligned
\]

In $\overline{\mathcal{U}^{\text{form}}_3}$, there is an integer $a \neq 0$ such that
\[
[K] 
-
[K']
=
a(\mathcal{T}_+ - \mathcal{T}_-).
\]
We deduce that a$\mathcal{T}_+$ is a subcomplex and a direct summand of $K$, and that $a\mathcal{T}_-$ is a direct summand of $K'$. Furthermore, the dimension of the component of bidegree $(1, 1)$ in $K$ is equal to $9 - \delta$, and the component of bidegree $(1, 1)$ in $K'$ has dimension $9 - \delta'$. We deduce that
\[
\delta' - \delta
=
-a.
\]
Clearly, we have $a = 1$ or $a= -1$ since $\delta, \delta' \in \{0, 1\}$. We may suppose that $a = 1$ and that $\delta = 1$ and $\delta' = 0$.

Now, since $\Lambda^{\bullet, \bullet}\mathfrak{g'}^*$ has a square in its decomposition, we know that $\Lambda^{\bullet, \bullet}\mathfrak{g'}^*$ has two dots in bidegree $(1, 0)$. Since this must also be the case for $\Lambda^{\bullet, \bullet}\mathfrak{g}^*$, we deduce that $d\omega^2=0$ also vanishes in its structure equations. 

$Z^{2, 2}_2$ is a direct summand in $\Lambda^{\bullet, \bullet}\mathfrak{g}^*$, therefore $\mathfrak{g}$ has a non-abelian nilpotent complex structure (corresponding to the right column, page $14$, in~\cite{COUV}). Since we also have $d\omega^2=0$, $\mathfrak{g}$ is isomorphic to $\mathfrak{h}_2$, $\mathfrak{h}_4$, $\mathfrak{h}_5$ or $\mathfrak{h}_6$, and we must have $B=0$, otherwise $Z^{2, 2}_2$ can not be a direct summand in their double complexes. We are left with the Lie algebras $\mathfrak{h}_2, \mathfrak{h}_4$ and $\mathfrak{h}_5$ where we are forced to take $\lambda = 0$ and $D \in \mathbb{R} \backslash \{1/2\}$ to have $Z^{2, 2}_2$ and $S_{1, 1}$ as direct summands. If $D \neq -1, 0$, we obtain the complex structure displayed on the left in~\ref{structures_egales}. For $D=-1, 0$, we show the double complex in~\ref{h5app}. 

Furthermore, amongst the Lie algebras $\mathfrak{h}_2$, $\mathfrak{h}_4$, $\mathfrak{h}_5$ and $\mathfrak{h}_6$, the only ones having a non-abelian nilpotent complex structure satisfying~\ref{no-square} are 
\begin{itemize}
\item $\mathfrak{h}_2$, with $\Re (D) = 1$
\item $\mathfrak{h}_4$, with $D=1$.
\item $\mathfrak{h}_5$, with $\lambda = 0$ and $\Re (D) = \frac{1}{2}$.
\end{itemize}
All of them have the complex structure displayed on the right in~\ref{structures_egales}.
\end{proof}

%%%%%%%%%%%%%%%%%%%%%%%%%%%%%%%%%%%%%%%%
\subsection{Almost abelian Lie algebras}
In this section we will consider {\sl almost abelian} Lie algebras.
\begin{Definition}
A Lie algebra $\mathfrak{g}$ is said to be {\sl almost abelian} if it has an abelian subalgebra $\mathfrak{a}$ of codimension $1$.
\end{Definition}
Consider such a Lie algebra $\mathfrak{g}$ of even dimension $2n+2$, and denote by $\mathfrak{a}$ an abelian subalgebra of codimension $1$. Choose a nonzero vector $e_0 \in \mathfrak{g} \backslash \mathfrak{a}$. As {\bf vector spaces}, $\mathfrak{g} \simeq \mathbb{R}e_0 \times \mathfrak{a}$. 

The space $\mathbb{R}e_0$ acts on $\mathfrak{a}$ with the Lie bracket:
\[
\lambda e_0 \cdot v
=
[\lambda e_0, v].
\]
This action is linear and can therefore be represented as a matrix $A$. Conversely, for any matrix $A$, one can define the Lie algebra $\mathbb{R}e_0 \ltimes_A \mathfrak{a}$ where the Lie bracket is given by
\[
[e_0, v] = Av,
\]
for all $v \in \mathfrak{a}$. 

The Lie algebra $\mathbb{R}e_0 \ltimes_A \mathfrak{a}$ is nilpotent if and only if $A$ is nilpotent, and two Lie algebras $\mathbb{R}e_0 \ltimes_A \mathfrak{a}$ and $\mathbb{R}e_0 \ltimes_A' \mathfrak{a}$ are isomorphic if and only if $A$ and $A'$ are conjugate.

Furthermore, as $A$ is nilpotent, each of its conjugacy classes have a representant given by a normal Jordan form. The existence of an almost complex structure on an almost abelian nilmanifold therefore only depends on the normal Jordan form of $A$. The following criterion is proven in~\cite{ABDGH}:
\begin{Theorem}
Suppose that the normal Jordan form of $A$ is given by the partition
\[
2n+1
=
m_1 * 1 + m_2 * 2 + m_3 * 3 + \cdots + m_k*k.
\]
Then the Lie algebra $\mathbb{R}e_0 \ltimes_A \mathfrak{a}$ admits an almost complex structure if and only if there is a partition of $n$:
\[
n = q_1 * 1 + q_2 * 2 + \cdots + q_k*k,
\]
and an integer $1 \leqslant j \leqslant k-1$, such that:
\[
\Bigg\{
\begin{array}{ccc}
m_{j-1} &= 2q_{j-1}-1 &\\
m_j &= 2q_j + 1 &\\
m_i &= 2q_i &\text{ for } i \neq j, j-1
\end{array}.
\]
\end{Theorem} 
The complex structures that exist on almost abelian Lie algebras were studied and classified in \cite{AABRW}.
Each almost abelian Lie algebra has in fact at most one complex structure, and when it exist, there is a basis 
\[
\mathcal{B}
=
\{
\alpha, \beta_1^0, \beta_2^0, \dots, \beta_{k_0}^0, \beta_1^1, \dots, \beta_{k_1}^1, \dots, \beta_{k_r}^r
\}
\]
of the $i$-eigenspace ${\mathfrak{g}^{1, 0}}^*$, satisfying
\begin{align}
\label{str_eq_ab}
& d\alpha = 0 \nonumber \\
&\Big\{
\begin{array}{cccc}
d\beta_1^0 &= &\alpha \wedge \overline{\alpha}, & \\
d\beta_i^0 &= &(\alpha + \overline{\alpha}) \wedge \beta_{i-1}^0, & \quad \text{for } i\geqslant 2.
\end{array} \\
&\Big\{
\begin{array}{cccc}
d\beta_1^j &= &0, & \\
d\beta_i^j &= &(\alpha + \overline{\alpha}) \wedge \beta_{i-1}^j, & \quad \text{for } i\geqslant 2,
\end{array}
\nonumber
\end{align}
with the last bracket being true for all $j \geqslant 1$. 

In \cite{AABRW}, an explicit computation of the Betti and Hodge numbers is given. They are given in terms of the number of irreducible components in some Lie algebra representations of $\mathfrak{sl}_2(\mathbb{C})$. We will investigate these representations in more details in the case $k_0=0$ to fully decompose the double complex associated to these almost abelian nilmanifolds.

We fix $k_0 = 0$ from now on. We will use the same notations as in \cite{AABRW}, and write $\mathfrak{b} = \mathcal{J}\mathfrak{a} \cap \mathfrak{a}$. A basis of $\mathfrak{b}^*$ is given by
\[
\{
\beta_1^1, \beta_2^1, \dots, \beta_{k_r}^r, \overline{\beta_1^1}, \dots, \overline{\beta_{k_r}^r}
\}.
\]
Now, the dual action of the nilpotent matrix $A$ on $\mathfrak{b}^*$ is 
\[
\aligned
A\beta_1^j &= 0, \\
A\beta_i^j &= \beta_{i-1}^j, \quad \text{for} \ i \geqslant 2,
\endaligned
\]
for all $j \geqslant 1$.
By the Jacobson-Morozov Theorem (\cite{Knapp}), the matrix $A$ extends to a $\mathfrak{sl}_2(\mathbb{C})$-triple $(H, A, N)$. Therefore, the eigenspaces ${\mathfrak{b}^*}^{1, 0}$ and ${\mathfrak{b}^*}^{0, 1}$ are representations of the Lie algebra $\mathfrak{sl}_2(\mathbb{C})$. For any integers $p$ and $q$, the space $\Lambda^{p, q} \mathfrak{b}^*$ of left invariant $(p, q)$-forms is also a representation of $\mathfrak{sl}_2(\mathbb{C})$,where the action of any matrix $M \in \mathfrak{sl}_2(\mathbb{C})$ is given by the Leibniz rule
\[
\aligned
M(\beta_{i_1}^{j_1} \wedge \cdots \wedge \beta_{i_p}^{j_p} \wedge \cdots \wedge \overline{\beta_{i_{p+q}}^{j_{p+q}}})
&=
\sum_{l=1}^{p} \beta_{i_1}^{j_1} \wedge \cdots \wedge M\beta_{i_l}^{j_l} \wedge \cdots \wedge \overline{\beta_{i_{p+q}}^{j_{p+q}}} \\
&+
\sum_{l=p+1}^{p+q} \beta_{i_1}^{j_1} \wedge \cdots \wedge M\overline{\beta_{i_l}^{j_l}} \wedge \cdots \wedge \overline{\beta_{i_{p+q}}^{j_{p+q}}}.
\endaligned
\]
In particular, this is true for $M = A$.
Furthermore, it is know that for each integer $n$ there is, up to isomorphism, exactly one irreducible $\mathfrak{sl}_2(\mathbb{C})$ representation $W_n$ of dimension $n$. This representation is isomorphic to the symmetric product
\[
W_n
=
Sym_n(\mathbb{C}^2),
\] 
and it is always possible to choose an element $v \in W_n$ such that the family
\[
\{
v, Av, A^2v, \cdots, A^{n-1}v
\}
\]
is a basis of $W_n$~(\cite{FH}).
Hence, $\Lambda^{p, q}\mathfrak{b}^* = \Lambda^p {\mathfrak{b}^*}^{1, 0} \otimes \Lambda^q {\mathfrak{b}^*}^{0, 1}$ can be decomposed as a sum of irreducible representations, and there exist unique integers $c_1^{p, q}, c_2^{p, q}, \dots$ such that:
\[
\Lambda^{p, q}\mathfrak{b}^*
=
c_1^{p, q} W_1
\oplus
c_2^{p, q} W_2
\oplus 
\cdots.
\]
Now, in every copy of $W_k$ appearing in this decomposition, we may choose an element $b_k^{p, q}$ such that the family
\[
\{
b_k^{p, q}, Ab_k^{p, q}, A^2b_k^{p, q}, \dots, A^{k-1}b_k^{p, q}
\}
\]
is a basis of $W_k$. Since $W_k$ appears $c_k^{p, q}$ times in the decomposition of $\Lambda^{p, q} \mathfrak{b}^*$, we have $c_k^{p, q}$ such elements, one in each copy of $W_k$, that we shall denote by
\[
b_{k, 1}^{p, q},
b_{k, 2}^{p, q},
\dots,
b_{k, c_k^{p, q}}^{p, q}.
\]
Hence, the space $\Lambda^{\bullet, \bullet} \mathfrak{g}^*$ of left-invariant forms admits the following basis:
\[
\{
A^lb_{k, s}^{p, q}
\}
\cup
\{
\alpha \wedge A^lb_{k, s}^{p, q}
\}
\cup
\{
\overline{\alpha} \wedge A^lb_{k, s}^{p, q}
\}
\cup
\{
\overline{\alpha} \wedge \alpha \wedge A^lb_{k, s}^{p, q}
\}.
\] 
Furthermore, by the definition of the differential $d = \partial + \overline{\partial}$, we see that
\[
\aligned
\partial \beta_i^j &= \alpha \wedge A\beta_i^j, \\
\overline{\partial} \beta_i^j &= \overline{\alpha} \wedge A\beta_i^j,
\endaligned
\]
for all $i, j \geqslant 0$.

Since $\partial$ and $\alpha \wedge A$ coincide on a basis of $\mathfrak{b}^*$, and since both satisfy the graded Leibniz rule, they coincide on the whole exterior algebra $\Lambda^{\bullet, \bullet}\mathfrak{b}^*$.
The same is true for $\overline{\partial}$ and $\overline{\alpha} \wedge A$.
\[
\aligned
\partial \omega &= \alpha \wedge A\omega, \\
\overline{\partial} \omega &= \overline{\alpha} \wedge A\omega.
\endaligned
\]

We will now use this explicit formula of the differentials $\partial$ and $\overline{\partial}$ in terms of $A$ to describe some subcomplexes of $\Lambda^{p, q} \mathfrak{b}^*$. Ultimately, we will see that these subcomplexes are direct summands, and therefore they describe the decomposition of the double complex into irreducible parts. 

We will consider the following odd zigzags of length $3$, squares and points:
\[
\aligned
Z_1^{p, q, k, s} = &\begin{tikzcd}[row sep=small, column sep = small]
\overline{\alpha} \wedge A^{k-1}b_{k, s}^{p, q}&  \\
A^{k-2}b_{k, s}^{p, q} \arrow[u] \arrow[r]& \alpha \wedge A^{k-1}b_{k, s}^{p, q}
\end{tikzcd}
&\text{For } 2 \leqslant k.
\\
\\
Z_2^{p, q, k, s} = &\begin{tikzcd}[row sep=small, column sep = small]
\overline{\alpha} \wedge b_{k, s}^{p-1, q-1} \arrow[r] & \overline{\alpha} \wedge \alpha \wedge Ab_{k, s}^{p-1, q-1}&  \\
& -\alpha \wedge b_{k, s}^{p-1, q-1} \arrow[u]
\end{tikzcd}.
\\
\\
Z_3^{p, q, k, s, l} = &\begin{tikzcd}[row sep=small, column sep = small]
\overline{\alpha} \wedge A^{l+1}b_{k, s}^{p, q} \arrow[r] & \overline{\alpha} \wedge \alpha \wedge A^{l+2}b_{k, s}^{p, q}&  \\
A^lb_{k, s}^{p, q} \arrow[r], \arrow[u] & \alpha \wedge A^{l+1}b_{k, s}^{p, q} \arrow[u]
\end{tikzcd}
&\text{For } 0 \leqslant l \leqslant k-3.
\\
\\
Z_4^{p, q, k, s} = &\quad A^{k-1}b_{k, s}^{p, q}.
\\
\\
Z_5^{p, q, s} = &\quad \alpha \wedge b_{1, s}^{p-1, q}.
\\
\\
Z_6^{p, q, s} = &\quad \overline{\alpha} \wedge b_{1, s}^{p, q-1}.
\\
\\
Z_7^{p, q, k, s} = &\quad \overline{\alpha} \wedge \alpha \wedge b_{k, s}^{p-1, q-1}.
\endaligned
\]
These irreducible double complexes are all subcomplexes of $\Lambda^{\bullet, \bullet}\mathfrak{g}^*$. We shall parametrize these subcomplexes by the set $I \subset \mathbb{N}^6$ of $6$-uples $(i, p, q, r, s, l)$ such that $Z_i^{p, q, r, s, l}$ is a subcomplex defined above. 
\begin{Property}
The elements of the family of subcomplexes
\[
\{
Z_i^{p, q, r, s, l}
\}_{(i, p, q, r, s, l) \in I}
\]
have zero intersection, and they sum up to the whole space:
\[
\bigoplus_{(i, p, q, r, s, l) \in I}
Z_i^{p, q, r, s, l}
=
\Lambda^{\bullet, \bullet}\mathfrak{g}^*.
\]
\end{Property}
\begin{proof}
Each element of the basis
\[
\{
A^lb_{k, s}^{p, q}
\}
\cup
\{
\alpha \wedge A^lb_{k, s}^{p, q}
\}
\cup
\{
\overline{\alpha} \wedge A^lb_{k, s}^{p, q}
\}
\cup
\{
\overline{\alpha} \wedge \alpha \wedge A^lb_{k, s}^{p, q}
\}
\]
of $\Lambda^{\bullet, \bullet}\mathfrak{g}^*$ appears exactly one time in one of the components of one of the element of the family
\[
\{
Z_i^{p, q, r, s, l}
\}_{(i, p, q, r, s, l) \in I}.
\]
\end{proof}

As a consequence of this decomposition, we can completely describe the zigzags which appears in the double complex of $\mathbb{R}e_0 \ltimes_A \mathfrak{a}$, in the case $j=1$ (equivalently $k_0=0$ in the structure equations).
\begin{Theorem}
\label{dec_ab}
The double complex associated to $\mathbb{R}e_0 \ltimes_A \mathfrak{a}$ with $k_0=j-1=0$  only contains dots, squares and odd zigzags of length $3$. Furthermore, the number of squares which have their lower left component in bidegree $(p, q)$ is equal to
\[
c_3^{p, q} + 2c_4^{p,q} + 3c_5^{p, q} + \cdots + (k-2)c_k^{p, q} + \cdots,
\]
and the number of odd zigzags which are looking up and have their lower left component in bidegree $(p, q)$ is equal to
\[
c_2^{p, q} + c_3^{p, q} + c_4^{p, q} + \cdots + c_k^{p, q} + \cdots.
\]
\end{Theorem}
As a consequence, we notice that the decomposition of such a double complex contains only squares and zigzags of length $1$ and $3$. Manifolds admitting a double complex with such a decomposition are referred to as $dd^c+3$-manifolds in~\cite{SW}.
\begin{Corollary}
Let $A$ be a nilpotent matrix, and $G$ the Lie group associated to the Lie algebra $\mathbb{R}e_0 \ltimes_A \mathfrak{a}$.
Assume that the Lie algebra $\mathbb{R}e_0 \ltimes_A \mathfrak{a}$ admits a complex structure such that $k_0=0$ in the structure equations~\ref{str_eq_ab}. 
Then, for any lattice $\Gamma$ of $G$, the compact complex manifold $G / \Gamma$ is $dd^c+3$.
\end{Corollary}
If $k_0 \neq 0$, this property no longer holds, and longer zigzags will have nonzero multiplicities.

For completeness, we shall give, for each puzzle piece in \eqref{pieces_puzzle}, their number of occurences in each bidegree $(p, q)$. Note that we are here not considering the squares.
\[
\aligned
\case{0}{0} &= 
\begin{array}{ccc}
&c_1^{p, q} + c_2^{p, q} + \cdots \\
&+c_1^{p-1, q} + c_1^{p, q-1} \\
&+c_1^{p-1, q-1} + c_2^{p-1, q-1} + \cdots ,
\end{array}
\\
\case{1}{0} &= c_2^{p-1, q-1} + c_3^{p-1, q-1} + c_4^{p-1, q-1} ,\\
\case{2}{0} &= c_2^{p, q-1} + c_3^{p, q-1} + \cdots ,\\
\case{3}{0} &= c_2^{p-1, q} + c_3^{p-1, q} + c_4^{p-1, q} + \cdots ,\\
\case{4}{0} &= c_2^{p-1, q} + c_3^{p-1, q} + \cdots ,\\
\case{5}{0} &= c_2^{p, q-1} + c_3^{p, q-1} + c_4^{p, q-1} + \cdots ,\\
\case{6}{0} &= c_2^{p, q} + c_3^{p, q} + \cdots .\\
\endaligned
\]
\begin{Example}
Some cohomologies can be computed directly, as for example the Dolbeault cohomology:
\[
\aligned
h^{p, q}
&=
\case{0}{0} + \case{3}{0} + \case{5}{0}
\\
&=
c_1^{p, q} + c_2^{p, q} + \cdots 
\\
&\quad +
c_1^{p-1, q} + c_2^{p-1, q} + \cdots
\\
&\quad +
c_1^{p, q-1} + c_2^{p, q-1} + \cdots
\\
&\quad +
c_1^{p-1, q-1} + c_2^{p-1, q-1} + \cdots.
\endaligned
\]
If, for a $\mathfrak{sl}_2(\mathbb{C})$-representation $E$, we denote by $\delta(E)$ the number of irreducible representations in the decomposition of $E$, then we have the nice formula
\[
h^{p, q}
=
\delta(\Lambda^{p, q}{\mathfrak{b}^{1, 0}}^*)
+
\delta(\Lambda^{p-1, q}{\mathfrak{b}^{1, 0}}^*)
+
\delta(\Lambda^{p, q-1}{\mathfrak{b}^{1, 0}}^*)
+
\delta(\Lambda^{p-1, q-1}{\mathfrak{b}^{1, 0}}^*),
\]
which is also proven in \cite{AABRW} in the general case where $k_0 \geq 0$.  Similar formulas can be written for the Bott-Chern and Aeppli cohomologies, in the case $k_0 = 0$:
\[
\aligned
h_{\text{BC}}^{p, q} &= c_1^{p, q} + c_2^{p, q} + \cdots & h_{\text{A}}^{p, q} &=  c_1^{p, q} + 2c_2^{p, q} + 2c_3^{p, q} + \cdots \\
&+ c_1^{p-1, q} + c_2^{p-1, q} + \cdots & &+ c_1^{p-1, q} + c_2^{p-1, q} + \cdots \\
&+ c_1^{p, q-1} + c_2^{p, q-1} + \cdots & &+ c_1^{p, q-1} + c_2^{p, q-1} + \cdots \\
&+ c_1^{p-1, q-1} + 2c_2^{p-1, q-1} +2c_3^{p-1, q-1} + \cdots & &+ c_1^{p-1, q-1} + c_2^{p-1, q-1} + \cdots 
\endaligned
\]
\end{Example}
\begin{Example}
We shall illustrate Theorem~\ref{dec_ab} in the case $k_0 = 0$ and $k_1 = 2$. As a $\mathfrak{sl}_2(\mathbb{C})$ representation, ${\mathfrak{b}^*}^{1, 0}$ is isomorphic to $W_2$. In the diagram below, we decompose the 
$\mathfrak{sl}_2(\mathbb{C})$ representations $\Lambda^{p, q}\mathfrak{b}^*$ on the left, 
and we show the double complex on the right.
\[
\begin{array}{cccc}
0 & 0 & 0 & 0 \\
W_1 & W_2 & W_1 & 0 \\
W_2 & W_1 \oplus W_3 & W_2 & 0 \\
W_1 & W_2 & W_1 & 0
\end{array}
\quad \quad
\begin{tikzpicture}[baseline=5.8ex, inner sep=0pt, outer sep=0pt, remember picture]
\draw (0, 0) -- (2, 0) -- (2, 2) -- (0, 2) -- (0, 0);
\draw (0, .5) -- (2, .5);
\draw (0, 1) -- (2, 1);
\draw (0, 1.5) -- (2, 1.5);
\draw (.5, 0) -- (.5, 2);
\draw (1, 0) -- (1, 2);
\draw (1.5, 0) -- (1.5, 2);
\draw (.35, 1.2) -- (.65, 1.2) -- (.65, .9);
\draw (1.2, .35) -- (1.2, .65) -- (.9, .65);
\draw (.8, 1.65) -- (.8, 1.35) -- (1.1, 1.35);
\draw (1.65, .8) -- (1.35, .8) -- (1.35, 1.1);
\filldraw (.2, .2) circle (1pt);
\filldraw (1.8, .2) circle (1pt);
\filldraw (.2, 1.8) circle (1pt);
\filldraw (1.8, 1.8) circle (1pt);
\filldraw (.1, .6) circle (1pt);
\filldraw (.1, .7) circle (1pt);
\filldraw (1.9, 1.4) circle (1pt);
\filldraw (1.9, 1.3) circle (1pt);
\filldraw (.6, .1) circle (1pt);
\filldraw (.7, .1) circle (1pt);
\filldraw (1.4, 1.9) circle (1pt);
\filldraw (1.3, 1.9) circle (1pt);
\filldraw (.1, 1.4) circle (1pt);
\filldraw (1.4, .1) circle (1pt);
\filldraw (1.9, .6) circle (1pt);
\filldraw (.6, 1.9) circle (1pt);
\draw (0.25, 1.15) -- (0.25, .75) -- (.6, .75);
\draw (1.75, .85) -- (1.75, 1.25) -- (1.4, 1.25);
\draw (.75, .6) -- (.75, .25) -- (1.15, .25);
\draw (1.25, 1.4) -- (1.25, 1.75) -- (.85, 1.75);
\draw (1, 1) circle (4pt);
\draw (.8, 1.15) -- (.8, .8) -- (1.15, .8);
\draw (.85, 1.2) -- (1.2, 1.2) -- (1.2, .85); 
\filldraw (.65, .65) circle (1pt);
\filldraw (.73, .73) circle (1pt);
\filldraw (.57, .57) circle (1pt);
\filldraw (1.35, 1.35) circle (1pt);
\filldraw (1.27, 1.27) circle (1pt);
\filldraw (1.43, 1.43) circle (1pt);
\filldraw (.6, 1.4) circle (1pt);
\filldraw (.6, 1.3) circle (1pt);
\filldraw (.7, 1.4) circle (1pt);
\filldraw (.7, 1.3) circle (1pt);
\filldraw (1.4, .6) circle (1pt);
\filldraw (1.4, .7) circle (1pt);
\filldraw (1.3, .6) circle (1pt);
\filldraw (1.3, .7) circle (1pt);
\filldraw (.2, .2) circle (1pt);
\filldraw (1.8, .2) circle (1pt);
\filldraw (.2, 1.8) circle (1pt);
\filldraw (1.8, 1.8) circle (1pt);
\end{tikzpicture}
\, .
\]
Consider now the case $k_0=0$ and $k_1=3$. Then 
${\mathfrak{b}^*}^{1, 0} = W_3$. Since $\Lambda^2 W_3 = W_3$ and 
$W_3 \otimes W_3 = W_1 \oplus W_3 \oplus W_5$, we have the following decompositions:
\[
\begin{array}{ccccc}
0 & 0 & 0 & 0 & 0 \\
& & & & \\
W_1 & W_3 & W_3 & W_1 & 0 \\
& & & & \\
W_3 & W_1 \oplus W_3 \oplus W_5 & W_1 \oplus W_3 \oplus W_5 & W_3 & 0 \\
& & & & \\
W_3 & W_1 \oplus W_3 \oplus W_5 & W_1 \oplus W_3 \oplus W_5 & W_3 & 0 \\
& & & & \\
W_1 & W_3 & W_3 & W_1 & 0
\end{array}
\quad
\begin{tikzpicture}[baseline=13ex, inner sep=0pt, outer sep=0pt, remember picture, scale=1.6]
\draw (0, 0) -- (2.5, 0) -- (2.5, 2.5) -- (0, 2.5) -- (0, 0);
\draw (0, .5) -- (2.5, .5);
\draw (0, 1) -- (2.5, 1);
\draw (0, 1.5) -- (2.5, 1.5);
\draw (0, 2) -- (2.5, 2);
\draw (.5, 0) -- (.5, 2.5);
\draw (1, 0) -- (1, 2.5);
\draw (1.5, 0) -- (1.5, 2.5);
\draw (2, 0) -- (2, 2.5);

\draw (.35, 1.1) -- (.35, .85) -- (.6, .85);
\draw (.4, 1.15) -- (.65, 1.15) -- (.65, .9);

\draw (.35, 1.6) -- (.35, 1.35) -- (.6, 1.35);
\draw (.4, 1.65) -- (.65, 1.65) -- (.65, 1.4);

\draw (1.1, .35) -- (.85, .35) -- (.85, .6);
\draw (1.15, .4) -- (1.15, .65) -- (.9, .65);

\draw (1.6, .35) -- (1.35, .35) -- (1.35, .6);
\draw (1.65, .4) -- (1.65, .65) -- (1.4, .65);

%%%%%%%%%%%%%%%%%%%%%%%%%%%%%%%%%%%%%
\draw (2.15, 1.4) -- (2.15, 1.65) -- (1.9, 1.65);
\draw (2.1, 1.35) -- (1.85, 1.35) -- (1.85, 1.6);

\draw (2.15, .9) -- (2.15, 1.15) -- (1.9, 1.15);
\draw (2.1, .85) -- (1.85, .85) -- (1.85, 1.1);

\draw (1.4, 2.15) -- (1.65, 2.15) -- (1.65, 1.9);
\draw (1.35, 2.1) -- (1.35, 1.85) -- (1.6, 1.85);

\draw (.9, 2.15) -- (1.15, 2.15) -- (1.15, 1.9);
\draw (.85, 2.1) -- (.85, 1.85) -- (1.1, 1.85);
%%%%%%%%%%%%%%%%%%%%%%%%%%%%%%%%%%%%%

\draw (.85, 1.1) -- (.85, .85) -- (1.1, .85);
\draw (.9, 1.15) -- (1.15, 1.15) -- (1.15, .9);

\draw (1.35, 1.1) -- (1.35, .85) -- (1.6, .85);
\draw (1.35, 1.15) -- (1.65, 1.15) -- (1.65, .9);

\draw (1.35, 1.6) -- (1.35, 1.35) -- (1.6, 1.35);
\draw (1.35, 1.65) -- (1.65, 1.65) -- (1.65, 1.4);

\draw (.85, 1.6) -- (.85, 1.35) -- (1.1, 1.35);
\draw (.9, 1.65) -- (1.15, 1.65) -- (1.15, 1.4);

%%%%%%%%%%%%%%%%%%%%%%%%%%%%%%%%%%%%%

\draw (.8, 1.15) -- (.8, .8) -- (1.15, .8);
\draw (.85, 1.2) -- (1.2, 1.2) -- (1.2, .85); 

\draw (1.3, 1.15) -- (1.3, .8) -- (1.65, .8);
\draw (1.35, 1.2) -- (1.7, 1.2) -- (1.7, .85);

\draw (1.3, 1.65) -- (1.3, 1.3) -- (1.65, 1.3);
\draw (1.35, 1.7) -- (1.7, 1.7) -- (1.7, 1.35);

\draw (.8, 1.65) -- (.8, 1.3) -- (1.15, 1.3);
\draw (.85, 1.7) -- (1.2, 1.7) -- (1.2, 1.35); 

%%%%%%%%%%%%%%%%%%%%%%%%%%%%%%%%%%%%%%%%%
\filldraw (.2, .2) circle (1pt);
\filldraw (2.3, .2) circle (1pt);
\filldraw (.1, 1.9) circle (1pt);
\filldraw (1.9, .1) circle (1pt);
\filldraw (.2, 2.3) circle (1pt);
\filldraw (2.3, 2.3) circle (1pt);
\filldraw (.1, .6) circle (1pt);
\filldraw (.1, .7) circle (1pt);
\filldraw (2.4, 1.9) circle (1pt);
\filldraw (2.4, 1.8) circle (1pt);
\filldraw (.6, .1) circle (1pt);
\filldraw (.7, .1) circle (1pt);
\filldraw (1.9, 2.4) circle (1pt);
\filldraw (1.8, 2.4) circle (1pt);
\filldraw (.1, 1.4) circle (1pt);
\filldraw (1.4, .1) circle (1pt);
\filldraw (2.4, 1.1) circle (1pt);
\filldraw (1.1, 2.4) circle (1pt);
\filldraw (2.4, .6) circle (1pt);
\filldraw (.6, 2.4) circle (1pt);
\draw[red,fill=white] (1, .5) circle (3pt);
\node[draw = none] at (1, .5)   (a) {1};
\draw[red,fill=white] (1.5, .5) circle (3pt);
\node[draw = none] at (1.5, .5)   (a) {1};
\draw[red,fill=white] (.5, 1) circle (3pt);
\node[draw = none] at (.5, 1)   (a) {1};
\draw[red,fill=white] (.5, 1.5) circle (3pt);
\node[draw = none] at (.5, 1.5)   (a) {1};
%%%%%%%%%%%%%%%%%%%%%%%%%%%%
\draw[red,fill=white] (1.5, 2) circle (3pt);
\node[draw = none] at (1.5, 2)   (a) {1};
\draw[red,fill=white] (1., 2) circle (3pt);
\node[draw = none] at (1., 2)   (a) {1};
\draw[red,fill=white] (2, 1.5) circle (3pt);
\node[draw = none] at (2, 1.5)   (a) {1};
\draw[red,fill=white] (2, 1) circle (3pt);
\node[draw = none] at (2, 1)   (a) {1};
%%%%%%%%%%%%%%%%%%%%%%%%%%%%
\draw[red,fill=white] (1, 1) circle (3pt);
\node[draw = none] at (1, 1)   (a) {4};
\draw[red,fill=white] (1, 1.5) circle (3pt);
\node[draw = none] at (1, 1.5)   (a) {4};
\draw[red,fill=white] (1.5, 1.5) circle (3pt);
\node[draw = none] at (1.5, 1.5)   (a) {4};
\draw[red,fill=white] (1.5, 1) circle (3pt);
\node[draw = none] at (1.5, 1)   (a) {4};
%%%%%%%%%%%%%%%%%%%%%%%%%%%%
\node[draw = none] at (.75, .75)   (a) {4};
\node[draw = none] at (1.25, .75)   (a) {5};
\node[draw = none] at (.75, 1.25)   (a) {5};
\node[draw = none] at (1.75, .75)   (a) {4};
\node[draw = none] at (.75, 1.75)   (a) {4};
\node[draw = none] at (1.25, 1.25)   (a) {8};
\node[draw = none] at (1.25, 1.75)   (a) {5};
\node[draw = none] at (1.75, 1.25)   (a) {5};
\node[draw = none] at (1.75, 1.75)   (a) {4};
\end{tikzpicture}.
\]
\end{Example}
%%%%%%%%%%%%%%%%%%%%%%%%%%%%%%%%%%%%%%%
\vspace{.3cm}
{\bf Acknowledgements:} I am deeply grateful to Dr.\ Stelzig from the LMU for introducing me to the research problems I worked on, and for guiding me through them. I am also grateful to Dr.\ Ugarte for sharing a preliminary version of the corrigendum to~\cite{COUV} on the classification of nilpotent complex structures on nilpotent Lie algebras of dimension $6$. I would also like to thank the Ecole Normale Supérieure and my ARPE academic supervisor Prof.\ Gabriele Facciolo for this extraordinary opportunity of doing an internship abroad, and the Ludwig-Maximilian Universität München for accepting this project and having me during this year.
%%%%%%%%%%%%%%%%%%%%%%%%%%%%%%%%%%%%%%%%
\newpage
\appendix
\section{Some decompositions of double complexes}
We present the decompositions of some non-abelian complex structures on the Lie algebras $h_2$, $h_4$, $h_5$ and $h_6$.

\subsection{$h_2$, $h_4$ and $h_6$ with $B=1$}
%%%%%%%%%%%%%%%%%%%%%%%%%%%%%%%%%%%%%
\[
\aligned
d \omega^1 &= 0, \\
d \omega^2 &= 0, \\
d \omega^3 &= \omega^{12} + \omega^{1 \overline{1}} + 
\omega^{1 \overline{2}} + D \omega^{2 \overline{2}},
\endaligned
\]
where $D = x + iy$ with $0 \leqslant y$.
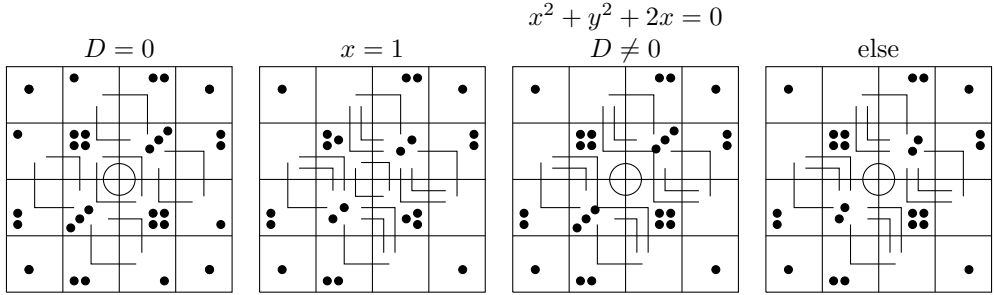
\[
\begin{array}{cccc}
& &  x^2 + y^2 + 2x = 0 & \\
D= 0 & x = 1 & D \neq 0 & \text{else} \\
\begin{tikzpicture}[baseline=5.8ex, inner sep=0pt, outer sep=0pt, remember picture, scale=1.5]
\draw (0, 0) -- (2, 0) -- (2, 2) -- (0, 2) -- (0, 0);
\draw (0, .5) -- (2, .5);
\draw (0, 1) -- (2, 1);
\draw (0, 1.5) -- (2, 1.5);
\draw (.5, 0) -- (.5, 2);
\draw (1, 0) -- (1, 2);
\draw (1.5, 0) -- (1.5, 2);

\draw (.35, 1.2) -- (.65, 1.2) -- (.65, .9);

\draw (1.2, .35) -- (1.2, .65) -- (.9, .65);

\draw (.8, 1.65) -- (.8, 1.35) -- (1.1, 1.35);

\draw (1.65, .8) -- (1.35, .8) -- (1.35, 1.1);
\filldraw (.2, .2) circle (1pt);
\filldraw (1.8, .2) circle (1pt);
\filldraw (.2, 1.8) circle (1pt);
\filldraw (1.8, 1.8) circle (1pt);
\filldraw (.1, .6) circle (1pt);
\filldraw (.1, .7) circle (1pt);
\filldraw (1.9, 1.4) circle (1pt);
\filldraw (1.9, 1.3) circle (1pt);
\filldraw (.6, .1) circle (1pt);
\filldraw (.7, .1) circle (1pt);
\filldraw (1.4, 1.9) circle (1pt);
\filldraw (1.3, 1.9) circle (1pt);
\filldraw (.1, 1.4) circle (1pt);
\filldraw (1.4, .1) circle (1pt);
\filldraw (1.9, .6) circle (1pt);
\filldraw (.6, 1.9) circle (1pt);
\draw (0.25, 1.15) -- (0.25, .75) -- (.6, .75);
\draw (1.75, .85) -- (1.75, 1.25) -- (1.4, 1.25);
\draw (.75, .6) -- (.75, .25) -- (1.15, .25);
\draw (1.25, 1.4) -- (1.25, 1.75) -- (.85, 1.75);
\draw (1, 1) circle (4pt);
\draw (.8, 1.15) -- (.8, .8) -- (1.15, .8);
\draw (.85, 1.2) -- (1.2, 1.2) -- (1.2, .85); 
\filldraw (.65, .65) circle (1pt);
\filldraw (.73, .73) circle (1pt);
\filldraw (.57, .57) circle (1pt);
\filldraw (1.35, 1.35) circle (1pt);
\filldraw (1.27, 1.27) circle (1pt);
\filldraw (1.43, 1.43) circle (1pt);
\filldraw (.6, 1.4) circle (1pt);
\filldraw (.6, 1.3) circle (1pt);
\filldraw (.7, 1.4) circle (1pt);
\filldraw (.7, 1.3) circle (1pt);
\filldraw (1.4, .6) circle (1pt);
\filldraw (1.4, .7) circle (1pt);
\filldraw (1.3, .6) circle (1pt);
\filldraw (1.3, .7) circle (1pt);
\filldraw (.2, .2) circle (1pt);
\filldraw (1.8, .2) circle (1pt);
\filldraw (.2, 1.8) circle (1pt);
\filldraw (1.8, 1.8) circle (1pt);
\end{tikzpicture}
&
\begin{tikzpicture}[baseline=5.8ex, inner sep=0pt, outer sep=0pt, remember picture, scale=1.5]
\draw (0, 0) -- (2, 0) -- (2, 2) -- (0, 2) -- (0, 0);
\draw (0, .5) -- (2, .5);
\draw (0, 1) -- (2, 1);
\draw (0, 1.5) -- (2, 1.5);
\draw (.5, 0) -- (.5, 2);
\draw (1, 0) -- (1, 2);
\draw (1.5, 0) -- (1.5, 2);
\draw (.35, 1.1) -- (.65, 1.1) -- (.65, .9);
\draw (.35, 1.2) -- (.75, 1.2) -- (.75, .9);
\draw (1.1, .35) -- (1.1, .65) -- (.9, .65);
\draw (1.2, .35) -- (1.2, .75) -- (.9, .75);
\draw (.9, 1.65) -- (.9, 1.35) -- (1.1, 1.35);
\draw (.8, 1.65) -- (.8, 1.25) -- (1.1, 1.25);
\draw (1.65, .9) -- (1.35, .9) -- (1.35, 1.1);
\draw (1.65, .8) -- (1.25, .8) -- (1.25, 1.1);
\filldraw (.2, .2) circle (1pt);
\filldraw (1.8, .2) circle (1pt);
\filldraw (.2, 1.8) circle (1pt);
\filldraw (1.8, 1.8) circle (1pt);
\filldraw (.1, .6) circle (1pt);
\filldraw (.1, .7) circle (1pt);
\filldraw (1.9, 1.4) circle (1pt);
\filldraw (1.9, 1.3) circle (1pt);
\filldraw (.6, .1) circle (1pt);
\filldraw (.7, .1) circle (1pt);
\filldraw (1.4, 1.9) circle (1pt);
\filldraw (1.3, 1.9) circle (1pt);
\draw (0.25, 1.15) -- (0.25, .75) -- (.6, .75);
\draw (1.75, .85) -- (1.75, 1.25) -- (1.4, 1.25);
\draw (.75, .6) -- (.75, .25) -- (1.15, .25);
\draw (1.25, 1.4) -- (1.25, 1.75) -- (.85, 1.75);
\filldraw (.6, 1.4) circle (1pt);
\filldraw (.6, 1.3) circle (1pt);
\filldraw (.7, 1.35) circle (1pt);
\filldraw (1.4, .6) circle (1pt);
\filldraw (1.4, .7) circle (1pt);
\filldraw (1.3, .65) circle (1pt);
\draw (.85, 1.1) -- (.85, .85) -- (1.1, .85);
\draw (.9, 1.15) -- (1.15, 1.15) -- (1.15, .9); 
\filldraw (.65, .65) circle (1pt);
\filldraw (.75, .75) circle (1pt);
\filldraw (1.35, 1.35) circle (1pt);
\filldraw (1.25, 1.25) circle (1pt);
\end{tikzpicture}
&
\begin{tikzpicture}[baseline=5.8ex, inner sep=0pt, outer sep=0pt, remember picture, scale=1.5]
\draw (0, 0) -- (2, 0) -- (2, 2) -- (0, 2) -- (0, 0);
\draw (0, .5) -- (2, .5);
\draw (0, 1) -- (2, 1);
\draw (0, 1.5) -- (2, 1.5);
\draw (.5, 0) -- (.5, 2);
\draw (1, 0) -- (1, 2);
\draw (1.5, 0) -- (1.5, 2);
\draw (0.25, 1.15) -- (0.25, .75) -- (.6, .75);
\draw (1.75, .85) -- (1.75, 1.25) -- (1.4, 1.25);
\draw (.75, .6) -- (.75, .25) -- (1.15, .25);
\draw (1.25, 1.4) -- (1.25, 1.75) -- (.85, 1.75);
\filldraw (.1, .6) circle (1pt);
\filldraw (.1, .7) circle (1pt);
\filldraw (1.9, 1.4) circle (1pt);
\filldraw (1.9, 1.3) circle (1pt);
\filldraw (.6, .1) circle (1pt);
\filldraw (.7, .1) circle (1pt);
\filldraw (1.4, 1.9) circle (1pt);
\filldraw (1.3, 1.9) circle (1pt);
\filldraw (.65, .65) circle (1pt);
\filldraw (.73, .73) circle (1pt);
\filldraw (.57, .57) circle (1pt);
\filldraw (1.35, 1.35) circle (1pt);
\filldraw (1.27, 1.27) circle (1pt);
\filldraw (1.43, 1.43) circle (1pt);
\draw (.35, 1.2) -- (.75, 1.2) -- (.75, .75);
\draw (1.2, .35) -- (1.2, .75) -- (.75, .75);
\draw (.8, 1.65) -- (.8, 1.25) -- (1.25, 1.25);
\draw (1.65, .8) -- (1.25, .8) -- (1.25, 1.25);
\draw (.35, 1.1) -- (.65, 1.1) -- (.65, .9);
\draw (1.1, .35) -- (1.1, .65) -- (.9, .65);
\draw (.9, 1.65) -- (.9, 1.35) -- (1.1, 1.35);
\draw (1.65, .9) -- (1.35, .9) -- (1.35, 1.1);
\filldraw (.6, 1.4) circle (1pt);
\filldraw (.6, 1.3) circle (1pt);
\filldraw (.7, 1.4) circle (1pt);
\filldraw (.7, 1.3) circle (1pt);
\filldraw (1.4, .6) circle (1pt);
\filldraw (1.4, .7) circle (1pt);
\filldraw (1.3, .6) circle (1pt);
\filldraw (1.3, .7) circle (1pt);
\draw (1, 1) circle (4pt);
\filldraw (.2, .2) circle (1pt);
\filldraw (1.8, .2) circle (1pt);
\filldraw (.2, 1.8) circle (1pt);
\filldraw (1.8, 1.8) circle (1pt);
\end{tikzpicture}
&
\begin{tikzpicture}[baseline=5.8ex, inner sep=0pt, outer sep=0pt, remember picture, scale=1.5]
\draw (0, 0) -- (2, 0) -- (2, 2) -- (0, 2) -- (0, 0);
\draw (0, .5) -- (2, .5);
\draw (0, 1) -- (2, 1);
\draw (0, 1.5) -- (2, 1.5);
\draw (.5, 0) -- (.5, 2);
\draw (1, 0) -- (1, 2);
\draw (1.5, 0) -- (1.5, 2);
\draw (0.25, 1.15) -- (0.25, .75) -- (.6, .75);
\draw (1.75, .85) -- (1.75, 1.25) -- (1.4, 1.25);
\draw (.75, .6) -- (.75, .25) -- (1.15, .25);
\draw (1.25, 1.4) -- (1.25, 1.75) -- (.85, 1.75);
\filldraw (.1, .6) circle (1pt);
\filldraw (.1, .7) circle (1pt);
\filldraw (1.9, 1.4) circle (1pt);
\filldraw (1.9, 1.3) circle (1pt);
\filldraw (.6, .1) circle (1pt);
\filldraw (.7, .1) circle (1pt);
\filldraw (1.4, 1.9) circle (1pt);
\filldraw (1.3, 1.9) circle (1pt);
\filldraw (.65, .65) circle (1pt);
\filldraw (.7, .75) circle (1pt);
\filldraw (1.35, 1.35) circle (1pt);
\filldraw (1.3, 1.25) circle (1pt);
\draw (.35, 1.1) -- (.65, 1.1) -- (.65, .9);
\draw (.35, 1.2) -- (.75, 1.2) -- (.75, .9);
\draw (1.1, .35) -- (1.1, .65) -- (.9, .65);
\draw (1.2, .35) -- (1.2, .75) -- (.9, .75);
\draw (.9, 1.65) -- (.9, 1.35) -- (1.1, 1.35);
\draw (.8, 1.65) -- (.8, 1.25) -- (1.1, 1.25);
\draw (1.65, .9) -- (1.35, .9) -- (1.35, 1.1);
\draw (1.65, .8) -- (1.25, .8) -- (1.25, 1.1);
\filldraw (.6, 1.4) circle (1pt);
\filldraw (.6, 1.3) circle (1pt);
\filldraw (.7, 1.4) circle (1pt);
\filldraw (.7, 1.3) circle (1pt);
\filldraw (1.4, .6) circle (1pt);
\filldraw (1.4, .7) circle (1pt);
\filldraw (1.3, .6) circle (1pt);
\filldraw (1.3, .7) circle (1pt);
\draw (1, 1) circle (4pt);
\filldraw (.2, .2) circle (1pt);
\filldraw (1.8, .2) circle (1pt);
\filldraw (.2, 1.8) circle (1pt);
\filldraw (1.8, 1.8) circle (1pt);
\end{tikzpicture}
\end{array}
\]
\vspace{.8cm}
%%%%%%%%%%%%%%%%%%%%%%%%%%%%%%%%%%%%%
\subsection{$h_2$, $h_4$ and $h_5$ with $B = 0$}
\label{h5app}
%%%%%%%%%%%%%%%%%%%%%%%%%%%%%%%%%%%%%
\[
\aligned
d \omega^1 &= 0, \\
d \omega^2 &= 0, \\
d \omega^3 &= \omega^{12} + \omega^{1 \overline{1}} + D \omega^{2 \overline{2}},
\endaligned
\]
where $D = x+iy$ with $0 \leqslant y$.
\[
\begin{array}{cccc}
D = -1 &
D = 0 & 
D = \frac{1}{2}\\
\begin{tikzpicture}[baseline=5.8ex, inner sep=0pt, outer sep=0pt, remember picture, scale=1.5]
\draw (0, 0) -- (2, 0) -- (2, 2) -- (0, 2) -- (0, 0);
\draw (0, .5) -- (2, .5);
\draw (0, 1) -- (2, 1);
\draw (0, 1.5) -- (2, 1.5);
\draw (.5, 0) -- (.5, 2);
\draw (1, 0) -- (1, 2);
\draw (1.5, 0) -- (1.5, 2);
\draw (0.25, 1.15) -- (0.25, .6) -- (.6, .6) -- (.6, .25) -- (1.15, .25);
\draw (1.75, .85) -- (1.75, 1.4) -- (1.4, 1.4) -- (1.4, 1.75) -- (.85, 1.75);
\filldraw (.1, .6) circle (1pt);
\filldraw (.1, .7) circle (1pt);
\filldraw (1.9, 1.4) circle (1pt);
\filldraw (1.9, 1.3) circle (1pt);
\filldraw (.6, .1) circle (1pt);
\filldraw (.7, .1) circle (1pt);
\filldraw (1.4, 1.9) circle (1pt);
\filldraw (1.3, 1.9) circle (1pt);
\filldraw (.6, .7) circle (1pt);
\filldraw (.6, .8) circle (1pt);
\filldraw (.6, .9) circle (1pt);
\filldraw (.75, .6) circle (1pt);
\filldraw (.85, .6) circle (1pt);
\filldraw (1.3, 1.4) circle (1pt);
\filldraw (1.2, 1.4) circle (1pt);
\filldraw (1.1, 1.4) circle (1pt);
\filldraw (1.4, 1.25) circle (1pt);
\filldraw (1.4, 1.15) circle (1pt);
\draw (.35, 1.1) -- (.65, 1.1) -- (.65, .65);
\draw (.35, 1.2) -- (.75, 1.2) -- (.75, .75);
\draw (1.1, .35) -- (1.1, .65) -- (.65, .65);
\draw (1.2, .35) -- (1.2, .75) -- (.75, .75);
\draw (.9, 1.65) -- (.9, 1.35) -- (1.35, 1.35);
\draw (.8, 1.65) -- (.8, 1.25) -- (1.25, 1.25);
\draw (1.65, .9) -- (1.35, .9) -- (1.35, 1.35);
\draw (1.65, .8) -- (1.25, .8) -- (1.25, 1.25);
\filldraw (.6, 1.4) circle (1pt);
\filldraw (.6, 1.3) circle (1pt);
\filldraw (.7, 1.4) circle (1pt);
\filldraw (.7, 1.3) circle (1pt);
\filldraw (1.4, .6) circle (1pt);
\filldraw (1.4, .7) circle (1pt);
\filldraw (1.3, .6) circle (1pt);
\filldraw (1.3, .7) circle (1pt);
\draw (1, 1) circle (4pt);
\filldraw (.2, .2) circle (1pt);
\filldraw (1.8, .2) circle (1pt);
\filldraw (.2, 1.8) circle (1pt);
\filldraw (1.8, 1.8) circle (1pt);
\end{tikzpicture}
&
\begin{tikzpicture}[baseline=5.8ex, inner sep=0pt, outer sep=0pt, remember picture, scale=1.5]
\draw (0, 0) -- (2, 0) -- (2, 2) -- (0, 2) -- (0, 0);
\draw (0, .5) -- (2, .5);
\draw (0, 1) -- (2, 1);
\draw (0, 1.5) -- (2, 1.5);
\draw (.5, 0) -- (.5, 2);
\draw (1, 0) -- (1, 2);
\draw (1.5, 0) -- (1.5, 2);
\draw (0.25, 1.15) -- (0.25, .6) -- (.6, .6) -- (.6, .25) -- (1.15, .25);
\draw (1.75, .85) -- (1.75, 1.4) -- (1.4, 1.4) -- (1.4, 1.75) -- (.85, 1.75);
\filldraw (.1, .6) circle (1pt);
\filldraw (.1, .7) circle (1pt);
\filldraw (1.9, 1.4) circle (1pt);
\filldraw (1.9, 1.3) circle (1pt);
\filldraw (.6, .1) circle (1pt);
\filldraw (.7, .1) circle (1pt);
\filldraw (1.4, 1.9) circle (1pt);
\filldraw (1.3, 1.9) circle (1pt);
\filldraw (.65, .65) circle (1pt);
\filldraw (.7, .75) circle (1pt);
\filldraw (.75, .65) circle (1pt);
\filldraw (1.35, 1.35) circle (1pt);
\filldraw (1.3, 1.25) circle (1pt);
\filldraw (1.25, 1.35) circle (1pt);
\draw (.35, 1.1) -- (.65, 1.1) -- (.65, .9);
\draw (1.1, .35) -- (1.1, .65) -- (.9, .65);
\draw (.9, 1.65) -- (.9, 1.35) -- (1.1, 1.35);
\draw (1.65, .9) -- (1.35, .9) -- (1.35, 1.1);
\filldraw (.6, 1.4) circle (1pt);
\filldraw (.6, 1.3) circle (1pt);
\filldraw (.7, 1.4) circle (1pt);
\filldraw (.7, 1.3) circle (1pt);
\filldraw (1.4, .6) circle (1pt);
\filldraw (1.4, .7) circle (1pt);
\filldraw (1.3, .6) circle (1pt);
\filldraw (1.3, .7) circle (1pt);
\draw (1, 1) circle (4pt);
\draw (.8, .85) -- (.8, 1.15);
\draw (1.2, .85) -- (1.2, 1.15);
\draw (.85, .8) -- (1.15, .8);
\draw (.85, 1.2) -- (1.15, 1.2);
\filldraw (.2, .2) circle (1pt);
\filldraw (1.8, .2) circle (1pt);
\filldraw (.2, 1.8) circle (1pt);
\filldraw (1.8, 1.8) circle (1pt);
\filldraw (.1, 1.4) circle (1pt);
\filldraw (1.4, .1) circle (1pt);
\filldraw (1.9, .6) circle (1pt);
\filldraw (.6, 1.9) circle (1pt);
\end{tikzpicture}
&
\begin{tikzpicture}[baseline=5.8ex, inner sep=0pt, outer sep=0pt, remember picture, scale=1.5]
\draw (0, 0) -- (2, 0) -- (2, 2) -- (0, 2) -- (0, 0);
\draw (0, .5) -- (2, .5);
\draw (0, 1) -- (2, 1);
\draw (0, 1.5) -- (2, 1.5);
\draw (.5, 0) -- (.5, 2);
\draw (1, 0) -- (1, 2);
\draw (1.5, 0) -- (1.5, 2);
\draw (0.25, 1.15) -- (0.25, .6) -- (.6, .6) -- (.6, .25) -- (1.15, .25);
\draw (1.75, .85) -- (1.75, 1.4) -- (1.4, 1.4) -- (1.4, 1.75) -- (.85, 1.75);
\draw (.85, 1.1) -- (.85, .85) -- (1.1, .85);
\draw (.9, 1.15) -- (1.15, 1.15) -- (1.15, .9); 
\filldraw (.1, .6) circle (1pt);
\filldraw (.1, .7) circle (1pt);
\filldraw (1.9, 1.4) circle (1pt);
\filldraw (1.9, 1.3) circle (1pt);
\filldraw (.6, .1) circle (1pt);
\filldraw (.7, .1) circle (1pt);
\filldraw (1.4, 1.9) circle (1pt);
\filldraw (1.3, 1.9) circle (1pt);
\filldraw (.65, .65) circle (1pt);
\filldraw (.7, .75) circle (1pt);
\filldraw (.75, .65) circle (1pt);
\filldraw (1.35, 1.35) circle (1pt);
\filldraw (1.3, 1.25) circle (1pt);
\filldraw (1.25, 1.35) circle (1pt);
\draw (.35, 1.1) -- (.65, 1.1) -- (.65, .9);
\draw (.35, 1.2) -- (.75, 1.2) -- (.75, .9);
\draw (1.1, .35) -- (1.1, .65) -- (.9, .65);
\draw (1.2, .35) -- (1.2, .75) -- (.9, .75);
\draw (.9, 1.65) -- (.9, 1.35) -- (1.1, 1.35);
\draw (.8, 1.65) -- (.8, 1.25) -- (1.1, 1.25);
\draw (1.65, .9) -- (1.35, .9) -- (1.35, 1.1);
\draw (1.65, .8) -- (1.25, .8) -- (1.25, 1.1);
\filldraw (.6, 1.4) circle (1pt);
\filldraw (.6, 1.3) circle (1pt);
\filldraw (.7, 1.35) circle (1pt);
\filldraw (1.4, .6) circle (1pt);
\filldraw (1.4, .7) circle (1pt);
\filldraw (1.3, .65) circle (1pt);
\filldraw (.2, .2) circle (1pt);
\filldraw (1.8, .2) circle (1pt);
\filldraw (.2, 1.8) circle (1pt);
\filldraw (1.8, 1.8) circle (1pt);
\end{tikzpicture}
\end{array}
\]
\vspace{.5cm}
\[
\begin{array}{cccc}
& D = \frac{1}{2} + iy &  \\
D \in \mathbb{R} \, \backslash \, \{-1, 0, \frac{1}{2}\}   & 0 < y < \sqrt{\frac{3}{4}} & \text{else} \\
\begin{tikzpicture}[baseline=5.8ex, inner sep=0pt, outer sep=0pt, remember picture, scale=1.5]
\draw (0, 0) -- (2, 0) -- (2, 2) -- (0, 2) -- (0, 0);
\draw (0, .5) -- (2, .5);
\draw (0, 1) -- (2, 1);
\draw (0, 1.5) -- (2, 1.5);
\draw (.5, 0) -- (.5, 2);
\draw (1, 0) -- (1, 2);
\draw (1.5, 0) -- (1.5, 2);
\draw (0.25, 1.15) -- (0.25, .6) -- (.6, .6) -- (.6, .25) -- (1.15, .25);
\draw (1.75, .85) -- (1.75, 1.4) -- (1.4, 1.4) -- (1.4, 1.75) -- (.85, 1.75);
\filldraw (.1, .6) circle (1pt);
\filldraw (.1, .7) circle (1pt);
\filldraw (1.9, 1.4) circle (1pt);
\filldraw (1.9, 1.3) circle (1pt);
\filldraw (.6, .1) circle (1pt);
\filldraw (.7, .1) circle (1pt);
\filldraw (1.4, 1.9) circle (1pt);
\filldraw (1.3, 1.9) circle (1pt);
\filldraw (.65, .65) circle (1pt);
\filldraw (.7, .75) circle (1pt);
\filldraw (.75, .65) circle (1pt);
\filldraw (1.35, 1.35) circle (1pt);
\filldraw (1.3, 1.25) circle (1pt);
\filldraw (1.25, 1.35) circle (1pt);
\draw (.35, 1.1) -- (.65, 1.1) -- (.65, .9);
\draw (.35, 1.2) -- (.75, 1.2) -- (.75, .9);
\draw (1.1, .35) -- (1.1, .65) -- (.9, .65);
\draw (1.2, .35) -- (1.2, .75) -- (.9, .75);
\draw (.9, 1.65) -- (.9, 1.35) -- (1.1, 1.35);
\draw (.8, 1.65) -- (.8, 1.25) -- (1.1, 1.25);
\draw (1.65, .9) -- (1.35, .9) -- (1.35, 1.1);
\draw (1.65, .8) -- (1.25, .8) -- (1.25, 1.1);
\filldraw (.6, 1.4) circle (1pt);
\filldraw (.6, 1.3) circle (1pt);
\filldraw (.7, 1.4) circle (1pt);
\filldraw (.7, 1.3) circle (1pt);
\filldraw (1.4, .6) circle (1pt);
\filldraw (1.4, .7) circle (1pt);
\filldraw (1.3, .6) circle (1pt);
\filldraw (1.3, .7) circle (1pt);
\draw (1, 1) circle (4pt);
\filldraw (.2, .2) circle (1pt);
\filldraw (1.8, .2) circle (1pt);
\filldraw (.2, 1.8) circle (1pt);
\filldraw (1.8, 1.8) circle (1pt);
\end{tikzpicture}
&
\begin{tikzpicture}[baseline=5.8ex, inner sep=0pt, outer sep=0pt, remember picture, scale=1.5]
\draw (0, 0) -- (2, 0) -- (2, 2) -- (0, 2) -- (0, 0);
\draw (0, .5) -- (2, .5);
\draw (0, 1) -- (2, 1);
\draw (0, 1.5) -- (2, 1.5);
\draw (.5, 0) -- (.5, 2);
\draw (1, 0) -- (1, 2);
\draw (1.5, 0) -- (1.5, 2);
\draw (.35, 1.1) -- (.65, 1.1) -- (.65, .9);
\draw (.35, 1.2) -- (.75, 1.2) -- (.75, .9);
\draw (1.1, .35) -- (1.1, .65) -- (.9, .65);
\draw (1.2, .35) -- (1.2, .75) -- (.9, .75);
\draw (.9, 1.65) -- (.9, 1.35) -- (1.1, 1.35);
\draw (.8, 1.65) -- (.8, 1.25) -- (1.1, 1.25);
\draw (1.65, .9) -- (1.35, .9) -- (1.35, 1.1);
\draw (1.65, .8) -- (1.25, .8) -- (1.25, 1.1);
\filldraw (.2, .2) circle (1pt);
\filldraw (1.8, .2) circle (1pt);
\filldraw (.2, 1.8) circle (1pt);
\filldraw (1.8, 1.8) circle (1pt);
\filldraw (.1, .6) circle (1pt);
\filldraw (.1, .7) circle (1pt);
\filldraw (1.9, 1.4) circle (1pt);
\filldraw (1.9, 1.3) circle (1pt);
\filldraw (.6, .1) circle (1pt);
\filldraw (.7, .1) circle (1pt);
\filldraw (1.4, 1.9) circle (1pt);
\filldraw (1.3, 1.9) circle (1pt);
\draw (0.25, 1.15) -- (0.25, .75) -- (.6, .75);
\draw (1.75, .85) -- (1.75, 1.25) -- (1.4, 1.25);
\draw (.75, .6) -- (.75, .25) -- (1.15, .25);
\draw (1.25, 1.4) -- (1.25, 1.75) -- (.85, 1.75);
\filldraw (.6, 1.4) circle (1pt);
\filldraw (.6, 1.3) circle (1pt);
\filldraw (.7, 1.35) circle (1pt);
\filldraw (1.4, .6) circle (1pt);
\filldraw (1.4, .7) circle (1pt);
\filldraw (1.3, .65) circle (1pt);
\draw (.85, 1.1) -- (.85, .85) -- (1.1, .85);
\draw (.9, 1.15) -- (1.15, 1.15) -- (1.15, .9); 
\filldraw (.65, .65) circle (1pt);
\filldraw (.75, .75) circle (1pt);
\filldraw (1.35, 1.35) circle (1pt);
\filldraw (1.25, 1.25) circle (1pt);
\end{tikzpicture}
&
\begin{tikzpicture}[baseline=5.8ex, inner sep=0pt, outer sep=0pt, remember picture, scale=1.5]
\draw (0, 0) -- (2, 0) -- (2, 2) -- (0, 2) -- (0, 0);
\draw (0, .5) -- (2, .5);
\draw (0, 1) -- (2, 1);
\draw (0, 1.5) -- (2, 1.5);
\draw (.5, 0) -- (.5, 2);
\draw (1, 0) -- (1, 2);
\draw (1.5, 0) -- (1.5, 2);
\draw (0.25, 1.15) -- (0.25, .75) -- (.6, .75);
\draw (1.75, .85) -- (1.75, 1.25) -- (1.4, 1.25);
\draw (.75, .6) -- (.75, .25) -- (1.15, .25);
\draw (1.25, 1.4) -- (1.25, 1.75) -- (.85, 1.75);
\filldraw (.1, .6) circle (1pt);
\filldraw (.1, .7) circle (1pt);
\filldraw (1.9, 1.4) circle (1pt);
\filldraw (1.9, 1.3) circle (1pt);
\filldraw (.6, .1) circle (1pt);
\filldraw (.7, .1) circle (1pt);
\filldraw (1.4, 1.9) circle (1pt);
\filldraw (1.3, 1.9) circle (1pt);
\filldraw (.65, .65) circle (1pt);
\filldraw (.7, .75) circle (1pt);
\filldraw (1.35, 1.35) circle (1pt);
\filldraw (1.3, 1.25) circle (1pt);
\draw (.35, 1.1) -- (.65, 1.1) -- (.65, .9);
\draw (.35, 1.2) -- (.75, 1.2) -- (.75, .9);
\draw (1.1, .35) -- (1.1, .65) -- (.9, .65);
\draw (1.2, .35) -- (1.2, .75) -- (.9, .75);
\draw (.9, 1.65) -- (.9, 1.35) -- (1.1, 1.35);
\draw (.8, 1.65) -- (.8, 1.25) -- (1.1, 1.25);
\draw (1.65, .9) -- (1.35, .9) -- (1.35, 1.1);
\draw (1.65, .8) -- (1.25, .8) -- (1.25, 1.1);
\filldraw (.6, 1.4) circle (1pt);
\filldraw (.6, 1.3) circle (1pt);
\filldraw (.7, 1.4) circle (1pt);
\filldraw (.7, 1.3) circle (1pt);
\filldraw (1.4, .6) circle (1pt);
\filldraw (1.4, .7) circle (1pt);
\filldraw (1.3, .6) circle (1pt);
\filldraw (1.3, .7) circle (1pt);
\draw (1, 1) circle (4pt);
\filldraw (.2, .2) circle (1pt);
\filldraw (1.8, .2) circle (1pt);
\filldraw (.2, 1.8) circle (1pt);
\filldraw (1.8, 1.8) circle (1pt);
\end{tikzpicture}
\end{array}
\]
%%%%%%%%%%%%%%%%%%%%%%%%%%%%%%%%%%%%%%%%

\end{document}